\documentclass[12pt,reqno]{amsart}

\usepackage[square,numbers]{natbib}              % BiBTeX style

\usepackage[centertags]{amsmath}        % AMS Stuff
\usepackage{amsfonts}
\usepackage{amssymb}
\usepackage{amsthm}
\usepackage{psfrag}                     % PSFrag (LaTeX style labels)
\usepackage{dcolumn}                    % Line up at decimal point in a table
\usepackage{graphicx}                   % Graphics stuff
\usepackage{color}                      % Colours
\usepackage[small]{caption}             % Captions
\usepackage[it]{subfig}                 % Subfigures
\usepackage{enumerate}                  % Enumerated lists
\usepackage{mathrsfs}                   % Serif fonts
\usepackage{longtable}                  % Stretch tables
\usepackage{lscape}                     % Landscape mode
\usepackage{afterpage}                  % Clear a page
\usepackage[figuresright]{rotating}     % Rotate a figure

\usepackage[all,cmtip]{xy}
\usepackage{comment}
\usepackage{mathdots}

\usepackage{mathrsfs}
\usepackage[all]{xy}
\DeclareMathAlphabet{\mathpzc}{OT1}{pzc}{m}{it}

\entrymodifiers={+!!<0pt,\fontdimen22\textfont2>}

\newcommand\xqed[1]{%
  \leavevmode\unskip\penalty9999 \hbox{}\nobreak\hfill
  \quad\hbox{#1}}
\newcommand\demo{\xqed{$\diamond$}}

\def\cD{\mathscr{D}}

\def\cJ{\mathscr{J}}

\def\cM{\mathscr{M}}

\def\cT{\mathscr{T}}

\def\sig{\Sigma^2}

\def\BN{\mathbb{N}}

\def\BW{\mathbb{W}}

\def\BZ{\mathbb{Z}}

\def\fE{\mathfrak{E}}

\def\fT{\mathfrak{T}}

\def\fY{\mathfrak{Y}}

\def\fa{\mathfrak{a}}
\def\fb{\mathfrak{b}}

\def\sD{\mathsf{D}}
\def\sE{\mathsf{E}}

\def\add{\operatorname{add}}

\def\adots{\mathinner{\mkern1mu\raise1.0pt\vbox{\kern7.0pt\hbox{.}}\mkern2mu\raise4.0pt\hbox{.}\mkern2mu\raise7.0pt\hbox{.}\mkern1mu}}

\def\ast{{\textstyle *}}

\def\colim{\operatorname{colim}}

\def\Gr{\operatorname{Gr}}
\def\dim{\operatorname{dim}}

\def\dddots{\mathinner{\mkern1mu\raise10.0pt\vbox{\kern7.0pt\hbox{.}}\mkern2mu\raise5.3pt\hbox{.}\mkern2mu\raise1.0pt\hbox{.}\mkern1mu}}
\def\dddotssmall{\mathinner{\mkern1mu\raise7.0pt\vbox{\kern7.0pt\hbox{.}}\mkern-1mu\raise4pt\hbox{.}\mkern-1mu\raise1.0pt\hbox{.}\mkern1mu}}

\def\dim{\operatorname{dim}}

\def\Gr{\mathsf{Gr}}

\def\Hom{\operatorname{Hom}}

\def\Image{\operatorname{Im}}

\def\ind{\operatorname{ind}}

\def\SL2{\operatorname{SL}_2}

\newcommand{\Dbar}{\overline{\mathscr{D}}}
\newcommand{\hmD}{\operatorname{Hom}_{\overline{\mathscr{D}}}}
\newcommand{\hocolim}{\operatorname{hocolim}}

\newtheorem{Lemma}{Lemma}[section]
\newtheorem{Theorem}[Lemma]{Theorem}
\newtheorem{Proposition}[Lemma]{Proposition}
\newtheorem{Corollary}[Lemma]{Corollary}

\theoremstyle{definition}
\newtheorem{Definition}[Lemma]{Definition}

\newtheorem{Remark}[Lemma]{Remark}

\newtheorem{Notation}[Lemma]{Notation}

\DeclareMathSymbol{\Gamma}{\mathalpha}{letters}{"00}
\DeclareMathSymbol{\Delta}{\mathalpha}{letters}{"01}
\DeclareMathSymbol{\Theta}{\mathalpha}{letters}{"02}
\DeclareMathSymbol{\Lambda}{\mathalpha}{letters}{"03}
\DeclareMathSymbol{\Xi}{\mathalpha}{letters}{"04}
\DeclareMathSymbol{\Pi}{\mathalpha}{letters}{"05}
\DeclareMathSymbol{\Sigma}{\mathalpha}{letters}{"06}
\DeclareMathSymbol{\Upsilon}{\mathalpha}{letters}{"07}
\DeclareMathSymbol{\Phi}{\mathalpha}{letters}{"08}
\DeclareMathSymbol{\Psi}{\mathalpha}{letters}{"09}
\DeclareMathSymbol{\Omega}{\mathalpha}{letters}{"0A}
\theoremstyle{definition}

\theoremstyle{definition}

\theoremstyle{definition}

\newsavebox{\astrutbox}
\sbox{\astrutbox}{\rule[-5pt]{0pt}{20pt}}

\begin{document}

\setlength{\parindent}{0pt}
\setlength{\parskip}{6pt}
%The default \baselineskip is close to 4.8mm
%\setlength{\baselineskip}{5.3mm}

\title[On the enlargement of the cluster category of type $A_\infty$]
{On the enlargement by Pr\"{u}fer objects of the cluster category of type $A_\infty$}

\author{Thomas A. Fisher}
\address{School of Mathematics and Statistics,
Newcastle University, Newcastle upon Tyne, NE1 7RU,
United Kingdom}
\email{t.a.fisher@ncl.ac.uk}

%\author{Next author goes here}
%\address{Next author's postal address goes here}
%\email{Next author's mail address goes here}

%\keywords{$2$-Calabi-Yau category, cluster structure, cluster tilting
%  subcategory, derived category, Differential Graded algebra,
%  Differential Graded module, Fomin-Zelevinsky mutation, spherical
%  object, quiver, weak cluster tilting subcategory}

%\subjclass[2000]{16E45, 16G20, 16G70}

%Our LaTeX packages do not permit the 2010 math subject label.  Hence
%the following construction.  We have also put the keywords in a \thanks
%to get the right order (subject class first, then keyworks).



%13F60: Cluster algebras
%16E10: Homological dimension
%16E45: Differential graded algebras and applications
%16G10: Representations of Artinian rings 
%16G20: Representations of quivers and partially ordered sets
%16G60: Representation type (finite, tame, wild, etc.) 
%16G70: Auslander-Reiten sequences (almost split sequences) and
%       Auslander-Reiten quivers
%18E30: Derived categories, triangulated categories
%18E35: Localization of categories 
%18G99: Homological algebra: None of the above, but in this section 
%55P62: Rational homotopy theory

\maketitle

\begin{abstract}
In \citep{J}, the cluster category $\cD$ of type $A_\infty$, with Auslander-Reiten quiver $\BZ A_\infty$, is introduced. Slices in the Auslander-Reiten quiver of $\cD$ give rise to direct systems; the homotopy colimit of such direct systems can be computed and these ``Pr\"{u}fer objects'' can be adjoined to form a larger category. It is this larger category, $\Dbar,$ which is the main object of study in this paper. We show that $\Dbar$ inherits a nice geometrical structure from $\cD$; ``arcs'' between non-neighbouring integers on the number line correspond to indecomposable objects, and in the case of $\Dbar$ we also have arcs to infinity which correspond to the Pr\"{u}fer objects. During the course of this paper, we show that $\Dbar$ is triangulated, compute homs, investigate the geometric model, and we conclude by computing the cluster tilting subcategories of $\Dbar$.
\end{abstract}

\setcounter{section}{-1}
\section{Introduction}
\label{sec:introduction}

In this paper, and in \citep{J}, $k$ is an algebraically closed field and $R = k[T]$ is viewed as a DG algebra with zero differential, with $T$ placed in cohomological degree $-1.$ In \citep{J}, the category $\cD = \sD^f(R),$ the derived category of (right) DG $R$-modules with finite-dimensional cohomology over $k,$ is introduced. 

It is a $k$-linear, Hom-finite, Krull-Schmidt, 2-Calabi-Yau triangulated category of algebraic origin. The indecomposables of $\cD$ are the objects $\Sigma^i X_n$ for $i \in \BZ,$ $n \in \BN_0$ where $X_n$ is obtained via a distinguished triangle
$$\xymatrix{
\Sigma^{n+1}R \ar[rr]^-{\cdot T^{n+1}} && R \ar[rr] && X_n,
}$$
see, for example, \citep[Section 1]{J}. The Auslander-Reiten quiver of $\cD$ is $\BZ A_\infty$. 

It was shown in \citep{J} that one can think of $\cD$ as a cluster category of type $A_\infty$. In particular, $\cD$ has a geometric model in terms of ``arcs'' between non-neighbouring integers on the number line, where a pair of integers $(a,b)$ (with $a \leq b-2$) corresponds to an indecomposable object, and the crossing of arcs corresponds to the non-vanishing of the $\operatorname{Ext}$-space between the corresponding indecomposables. 

It is in fact the sets of arcs which do not cross, which give rise to nicer constructions still: a maximal, non-crossing set of arcs is viewed as a triangulation of the $\infty$-gon, and the objects in the additive hull of the corresponding indecomposables will form a weakly cluster tilting subcategory. If the arcs are configured in a certain way it will turn out that the subcategory formed will be cluster tilting, i.e. it is also functorially finite. 

Slices in $\BZ A_\infty$ give rises to homotopy colimits - so-called ``Pr\"{u}fer objects'', see, for example, \citep[Introduction]{BABHK}. The purpose of this paper is to study the category $\Dbar$, where these Pr\"{u}fer objects have been adjoined to $\cD$. Accordingly, we show that $\Dbar$ is a $k$-linear, Hom-finite, Krull-Schmidt, triangulated category and we compute homs, extend the geometric model with arcs to infinity (where the crossing of arcs still corresponds to the non-vanishing of $\operatorname{Ext}$-spaces), and the paper concludes by finding the cluster tilting subcategories of $\Dbar$.

Note that apart from \citep{BABHK}, there are some other papers in the literature which study various incarnations of Pr\"{u}fer objects, see for instance \citep[Sec. 1]{AS} and \citep[Sec. 3.3]{BBM}.

\vspace{-16.5mm}

\section{Introducing the category $\overline{\cD}$}
\label{sec:introducingDbar}

\setlength{\parindent}{0pt}
\setlength{\parskip}{4pt}

In this section we introduce the category $\Dbar.$ We aim to define $\Dbar$ and conclude with proving it is triangulated.

\begin{Definition} \label{def:arquiver}
The Auslander-Reiten quiver of $\cD$ has a vertex for each indecomposable object $\Sigma^i X_n$ and an arrow between vertices if there is an irreducible morphism between the corresponding objects. The Auslander-Reiten quiver of $\cD$ is illustrated below, see \citep[Remark 1.4]{J}.
\[ \def\objectstyle{\scriptstyle}
  \xymatrix @-2.2pc @! {
    & \vdots \ar[dr] & & \vdots \ar[dr] & & \vdots \ar[dr] & & \vdots \ar[dr] & & \vdots \ar[dr] & & \vdots & \\
    \cdots \ar[dr]& & \Sigma^0 X_3 \ar[ur] \ar[dr] & & \Sigma^{-1} X_3 \ar[ur] \ar[dr] & & \Sigma^{-2} X_3 \ar[ur] \ar[dr] & & \Sigma^{-3}X_3 \ar[ur] \ar[dr] & & \Sigma^{-4}X_3 \ar[ur] \ar[dr] & & \cdots \\
    & \Sigma^1 X_2 \ar[ur] \ar[dr] & & \Sigma^0 X_2 \ar[ur] \ar[dr] & & \Sigma^{-1} X_2 \ar[ur] \ar[dr] & & \Sigma^{-2}X_2 \ar[ur] \ar[dr] & & \Sigma^{-3}X_2 \ar[ur] \ar[dr] & & \Sigma^{-4}X_2 \ar[ur] \ar[dr] & \\
    \cdots \ar[ur]\ar[dr]& & \Sigma^1 X_1 \ar[ur] \ar[dr] & & \Sigma^0 X_1 \ar[ur] \ar[dr] & & \Sigma^{-1} X_1 \ar[ur] \ar[dr] & & \Sigma^{-2}X_1 \ar[ur] \ar[dr] & & \Sigma^{-3}X_1 \ar[ur] \ar[dr] & & \cdots\\
    & \Sigma^2 X_0 \ar[ur] & & \Sigma^1 X_0 \ar[ur] & & \Sigma^0 X_0 \ar[ur] & & \Sigma^{-1}X_0 \ar[ur] & & \Sigma^{-2}X_0 \ar[ur] & & \Sigma^{-3}X_0 \ar[ur] & \\
               }
\] \demo
\end{Definition}

\begin{Definition} \label{def:slicedef}
A \textit{slice} in the Auslander-Reiten quiver of $\cD$ is a collection of vertices and arrows associated with the direct system of irreducible morphisms $$ \Sigma^nX_0 \rightarrow \Sigma^{n-1}X_1 \rightarrow \cdots \rightarrow \Sigma^{n-i}X_i \rightarrow \cdots.$$ Note that by \citep[Proposition 2.2]{J} the hom-space between two neighbouring objects in the system is 1-dimensional, so the system is determined up to isomorphism. The integer $n$ tells us where along the baseline the slice starts. For example, for $n=0$, the corresponding slice is illustrated below.
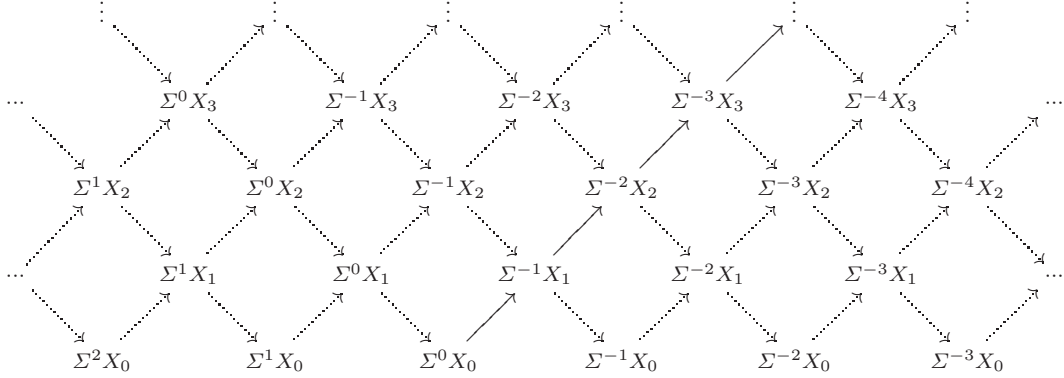
\[ \def\objectstyle{\scriptstyle}
  \xymatrix @-2.2pc @! {
    & \vdots \ar@{.>}[dr] & & \vdots \ar@{.>}[dr] & & \vdots \ar@{.>}[dr] & & \vdots \ar@{.>}[dr] & & \vdots \ar@{.>}[dr] & & \vdots & \\
    \cdots \ar@{.>}[dr]& & \Sigma^0 X_3 \ar@{.>}[ur] \ar@{.>}[dr] & & \Sigma^{-1} X_3 \ar@{.>}[ur] \ar@{.>}[dr] & & \Sigma^{-2} X_3 \ar@{.>}[ur] \ar@{.>}[dr] & & \Sigma^{-3}X_3 \ar[ur] \ar@{.>}[dr] & & \Sigma^{-4}X_3 \ar@{.>}[ur] \ar@{.>}[dr] & & \cdots \\
    & \Sigma^1 X_2 \ar@{.>}[ur] \ar@{.>}[dr] & & \Sigma^0 X_2 \ar@{.>}[ur] \ar@{.>}[dr] & & \Sigma^{-1} X_2 \ar@{.>}[ur] \ar@{.>}[dr] & & \Sigma^{-2}X_2 \ar[ur] \ar@{.>}[dr] & & \Sigma^{-3}X_2 \ar@{.>}[ur] \ar@{.>}[dr] & & \Sigma^{-4}X_2 \ar@{.>}[ur] \ar@{.>}[dr] & \\
    \cdots \ar@{.>}[ur]\ar@{.>}[dr]& & \Sigma^1 X_1 \ar@{.>}[ur] \ar@{.>}[dr] & & \Sigma^0 X_1 \ar@{.>}[ur] \ar@{.>}[dr] & & \Sigma^{-1} X_1 \ar[ur] \ar@{.>}[dr] & & \Sigma^{-2}X_1 \ar@{.>}[ur] \ar@{.>}[dr] & & \Sigma^{-3}X_1 \ar@{.>}[ur] \ar@{.>}[dr] & & \cdots\\
    & \Sigma^2 X_0 \ar@{.>}[ur] & & \Sigma^1 X_0 \ar@{.>}[ur] & & \Sigma^0 X_0 \ar[ur] & & \Sigma^{-1}X_0 \ar@{.>}[ur] & & \Sigma^{-2}X_0 \ar@{.>}[ur] & & \Sigma^{-3}X_0 \ar@{.>}[ur] & \\
               }
\]
The object $\Sigma^0X_0$ is the start of the slice, and is located on the baseline. The associated direct system is
$$\Sigma^0X_0 \rightarrow \Sigma^{-1}X_1 \rightarrow \cdots \rightarrow \Sigma^{-i}X_i \rightarrow \cdots.$$ \demo
\end{Definition}

\noindent The homotopy colimit, or hocolimit for short, of such direct systems is defined on page 209 of \citep{BN}.

\begin{Notation}\label{not:choice}
We will sometimes write $\hocolim_i(\Sigma^{n-i}X_i)$ as $\sE_n.$ Note that $\sE_n$ is in $\sD(R)$ but not in $\cD.$ \demo
\end{Notation}

\begin{Remark} \label{rem:onlyonehoc}
The only hocolimits which can be built from the indecomposables of $\cD$ are the hocolimits of direct systems of the form 
$$\xymatrix{
\Sigma^nX_0 \ar[r]^-{\xi_0} & \Sigma^{n-1}X_1 \ar[r]^-{\xi_1} & \Sigma^{n-2}X_2 \ar[r]& \cdots.
}
$$
This is because the only other way to obtain a direct system in the quiver would be to move up (and down) in the quiver infinitely often, like in a zig-zag. This is illustrated below with an example.
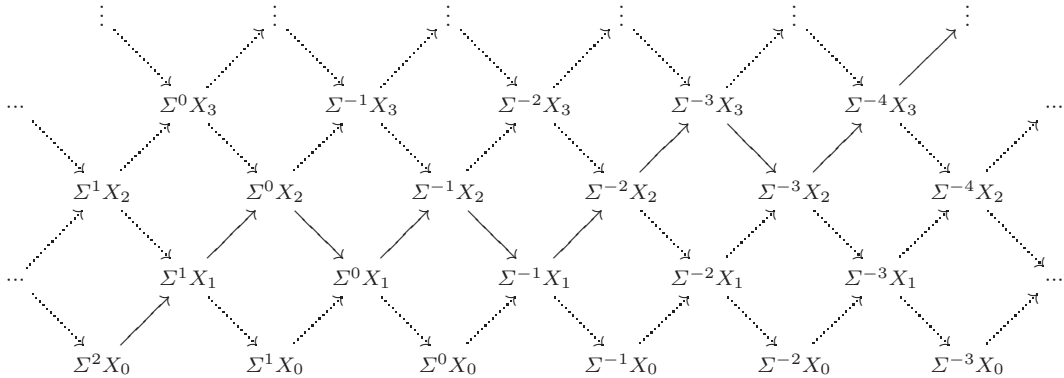
\[ \def\objectstyle{\scriptstyle}
  \xymatrix @-2.2pc @! {
    & \vdots \ar@{.>}[dr] & & \vdots \ar@{.>}[dr] & & \vdots \ar@{.>}[dr] & & \vdots \ar@{.>}[dr] & & \vdots \ar@{.>}[dr] & & \vdots & \\
    \cdots \ar@{.>}[dr]& & \Sigma^0 X_3 \ar@{.>}[ur] \ar@{.>}[dr] & & \Sigma^{-1} X_3 \ar@{.>}[ur] \ar@{.>}[dr] & & \Sigma^{-2} X_3 \ar@{.>}[ur] \ar@{.>}[dr] & & \Sigma^{-3}X_3 \ar@{.>}[ur] \ar[dr] & & \Sigma^{-4}X_3 \ar[ur] \ar@{.>}[dr] & & \cdots \\
    & \Sigma^1 X_2 \ar@{.>}[ur] \ar@{.>}[dr] & & \Sigma^0 X_2 \ar@{.>}[ur] \ar[dr] & & \Sigma^{-1} X_2 \ar@{.>}[ur] \ar[dr] & & \Sigma^{-2}X_2 \ar[ur] \ar@{.>}[dr] & & \Sigma^{-3}X_2 \ar[ur] \ar@{.>}[dr] & & \Sigma^{-4}X_2 \ar@{.>}[ur] \ar@{.>}[dr] & \\
    \cdots \ar@{.>}[ur]\ar@{.>}[dr]& & \Sigma^1 X_1 \ar[ur] \ar@{.>}[dr] & & \Sigma^0 X_1 \ar[ur] \ar@{.>}[dr] & & \Sigma^{-1} X_1 \ar[ur] \ar@{.>}[dr] & & \Sigma^{-2}X_1 \ar@{.>}[ur] \ar@{.>}[dr] & & \Sigma^{-3}X_1 \ar@{.>}[ur] \ar@{.>}[dr] & & \cdots\\
    & \Sigma^2 X_0 \ar[ur] & & \Sigma^1 X_0 \ar@{.>}[ur] & & \Sigma^0 X_0 \ar@{.>}[ur] & & \Sigma^{-1}X_0 \ar@{.>}[ur] & & \Sigma^{-2}X_0 \ar@{.>}[ur] & & \Sigma^{-3}X_0 \ar@{.>}[ur] & \\
               }
\]
Notice that if we stop moving down in the quiver, we necessarily end up on one of the direct systems of the form
$$\xymatrix{
\Sigma^nX_0 \ar[r]^-{\xi_0} & \Sigma^{n-1}X_1 \ar[r]^-{\xi_1} & \Sigma^{n-2}X_2 \ar[r]& \cdots
}
$$
and the zig-zagging at the beginning has no influence on its hocolimit. If we have a direct system where zig-zagging occurs infinitely often, then the results of this paper will show that the objects in the direct system won't have any nonzero maps to the hocolimit. This means that the hocolimit of such a direct system is zero. \demo
\end{Remark}

\begin{Definition}
Let $\Dbar$ be the category ``built'' from $\cD$ and all hocolimits of direct systems of the form
$$\xymatrix{
\Sigma^nX_0 \ar[r]^-{\xi_0} & \Sigma^{n-1}X_1 \ar[r]^-{\xi_1} & \Sigma^{n-2}X_2 \ar[r]& \cdots.
}
$$
Formally, we have 
$$\Dbar = \add\{\Sigma^i X_n, \sE_m\}$$ 
for $n \in \BN_0$ and $i, m \in \BZ,$ where $\add$ is taken inside $\sD(R)$. \demo
\end{Definition}

\begin{Remark}
Let $\sE_n$ be a hocolimit of the form in Notation \ref{not:choice}. Then $\Sigma^t \sE_n = \sE_{n+t},$ because it is obtained by applying the functor $\Sigma^t$ to the direct system (\ref{equ:slice}),
$$
\xymatrix{
\Sigma^nX_0 \ar[r]^-{\xi_0} & \Sigma^{n-1}X_1 \ar[r]^-{\xi_1} & \Sigma^{n-2}X_2 \ar[r]& \cdots,
}
$$
resulting in
$$
\xymatrix{
\Sigma^{n+t}X_0 \ar[r]^-{\Sigma^t \xi_0} & \Sigma^{n+t-1}X_1 \ar[r]^-{\Sigma^t \xi_1} & \Sigma^{n+t-2}X_2 \ar[r]& \cdots,
}
$$
which has hocolimit equal to $\sE_{n+t}$ in the category $\Dbar.$ \demo
\end{Remark}

\noindent Before computing morphisms in the category $\Dbar,$ we first prove that $\Dbar$ is triangulated.

\begin{Remark} \label{rem:different*}
There are two uses in this paper of the notation $(-)^{\ast}$. One use is to denote the dual of a vector space. Another use is to represent a grading. We aim to make it very clear which one is being used by the context. \demo
\end{Remark}

\begin{Remark} \label{rem:different(,)}
There are two uses in this paper of the notation $(-,-)$. One use is to denote the hom-space between two objects, like in $\Hom(X,Y)$. Another use is to represent an arc between two non-neighbouring integers on the number line, which we will see later on. It is always clear from the context which one is being used. \demo
\end{Remark}

\begin{Remark} \label{rem:thereisafunctor}
There is a functor 
$$F(-) = (R,\Sigma^{\ast}(-)) : \sD(R) \rightarrow \operatorname{Gr}((R,\Sigma^{\ast}R))$$ 
$$X \mapsto (R,\Sigma^{\ast}X)$$
from $\sD(R)$ to the category of graded right-modules over the graded $k$-algebra $(R,\Sigma^{\ast}R).$ We know $\Hom_{\sD(R)}(R,\text{---}) \cong H^0(\text{---})$ and this gives the first of the following isomorphisms: 
$$(R,\Sigma^{\ast}R) \cong H^{\ast}(R) \cong k[T].$$ The second holds since $R$ is just $k[T]$ equipped with the zero differential. Hence we have $\operatorname{Gr}((R,\Sigma^{\ast}R)) \cong \operatorname{Gr}(k[T]).$ Scalar multiplication in $(R,\Sigma^{\ast}R)$ works in the following way: if $m \in (R,\Sigma^iX), \hspace{1.5mm} a \in (R,\Sigma^jR),$ then $ma = \Sigma^j(m)\circ a.$ \demo
\end{Remark}

\noindent It will be convenient to abuse notation slightly and sometimes think of the functor $F$ being a functor $$F(-) : \sD(R) \rightarrow \Gr(k[T]),$$ due to the equivalence of categories $\operatorname{Gr}((R,\Sigma^{\ast}R)) \cong \operatorname{Gr}(k[T]).$

\begin{Lemma} \label{lem:Fishom}
The functor $F$ is homological; that is, if $X \rightarrow Y \rightarrow Z \rightarrow \Sigma X$ is a distinguished triangle in $\sD(R),$ then $F(X) \rightarrow F(Y) \rightarrow F(Z) \rightarrow F(\Sigma X)$ is an exact sequence.
\end{Lemma}

\begin{proof}
Apply $F$ to the distinguished triangle $\xymatrix@1{ X \ar[r]^-{f} & Y \ar[r]^-{g} & Z \ar[r]^-{h} & \Sigma X}$ to obtain
\begin{equation} \label{equ:fes}
  \xymatrix {
(R,\Sigma^{\ast}X) \ar[rr]^-{(R,\Sigma^{\ast}f)} && (R,\Sigma^{\ast}Y) \ar[rr]^-{(R,\Sigma^{\ast}g)} && (R,\Sigma^{\ast}Z) \ar[rr]^-{(R,\Sigma^{\ast}h)} && (R,\Sigma^{\ast+1} X).
}
\end{equation}
In degree $i,$ this is 
\[ 
  \xymatrix {
(R,\Sigma^{i}X) \ar[rr]^-{(R,\Sigma^{i}f)} && (R,\Sigma^{i}Y) \ar[rr]^-{(R,\Sigma^{i}g)} && (R,\Sigma^{i}Z) \ar[rr]^-{(R,\Sigma^{i}h)} && (R,\Sigma^{i+1} X).
}
\]
This sequence is exact, because $\Sigma^i X \rightarrow \Sigma^i Y \rightarrow \Sigma^i Z \rightarrow \Sigma^{i+1} X$ is again a distinguished triangle and $\Hom_{\sD(R)}(R,\text{---})$ is homological. Therefore, the sequence (\ref{equ:fes}) is an exact sequence, because it is exact at every degree.
\end{proof}

\noindent If $M$ is a graded module then $\Sigma M$ denotes the shift, so $(\Sigma M)^i = M^{i+1}.$

\begin{Lemma} \label{lem:ex34a}
If $\cM \subseteq \operatorname{Gr}(k[T])$ is a full subcategory closed under extensions, kernels, and cokernels, and $\Sigma \cM = \cM,$ then $\cJ = F^{-1}\cM$ is a triangulated subcategory of $\sD(R).$
\end{Lemma}

\begin{proof}
Let $X \rightarrow Y$ be a morphism in $\cJ.$ Then there exists an object $Z$ in $\sD(R)$ such that $X \rightarrow Y \rightarrow Z \rightarrow \Sigma X$ is a distinguished triangle. Extend this to 
$$X \rightarrow Y \rightarrow Z \rightarrow \Sigma X \rightarrow \Sigma Y$$ 
and apply $F$ to obtain the exact sequence
$$FX \rightarrow FY \rightarrow FZ \rightarrow F\Sigma X \rightarrow F\Sigma Y.$$
Because $F\Sigma \cong \Sigma F,$ this is isomorphic to
$$FX \rightarrow FY \rightarrow FZ \rightarrow \Sigma FX \rightarrow \Sigma FY.$$
It is clear that $FX \rightarrow FY$ is a morphism in $\cM$ and $\Sigma FX \rightarrow \Sigma FY$ is a morphism in $\Sigma \cM.$ But because $\cM = \Sigma \cM,$ both $FX \rightarrow FY$ and $\Sigma FX \rightarrow \Sigma FY$ are morphisms in $\cM.$ So let $C$ be the cokernel object of the map $FX \rightarrow FY,$ and $K$ the kernel object of the map $\Sigma FX \rightarrow \Sigma FY.$ Then
$$0 \rightarrow C \rightarrow FZ \rightarrow K \rightarrow 0$$
is a short exact sequence in $\operatorname{Gr}(k[T]).$ Now, both $C$ and $K$ are in $\cM$ because $\cM$ is closed under kernels and cokernels. This means that $FZ$ is in $\cM$ as well, and hence $Z \in F^{-1}\cM,$ because $\cM$ is closed under extensions. Hence, if there are two objects and a morphism between them in $F^{-1}\cM,$ then these objects complete to a distinguished triangle, and $\cJ = F^{-1}\cM$ is a triangulated subcategory of $\sD(R).$ We also have that $\Sigma^{\pm 1} \cJ \subseteq \cJ$ because $F(\Sigma^{\pm 1} \cJ) = \Sigma^{\pm 1} F(\cJ) = \Sigma^{\pm 1}\cM = \cM.$
\end{proof}

\begin{Definition} \label{def:locallyfinitemodule}
Let $M$ be an object in $\sD(R)$. Then $M$ is a graded module, say, 
$$M = \displaystyle{\bigoplus_{n \in \BZ}} M_n.$$ 
We say that $M$ is a {\textit{locally finite}} module if $\dim_k M_n < \infty$ for each $n \in \BZ$. \demo
\end{Definition}

\begin{Corollary} \label{lem:ex34}
We have that $F^{-1}((\operatorname{gr}(k[T]))^{\ast})$ is a triangulated subcategory of $\sD(R)$. 
\end{Corollary}

\begin{proof}
Let $\cM = (\operatorname{gr}(k[T]))^{\ast}$ where $\operatorname{gr}$ denotes the category of finitely generated modules. We have that $\operatorname{gr}(k[T])$ is closed under subobjects, quotient objects and extensions (and hence, kernels and cokernels). This implies that $\cM$ is also closed under subobjects, quotient objects and extensions (and hence, kernels and cokernels). This is because $(-)^{\ast}$ is a duality from locally finite modules to itself. Also, $\cM$ is a full subcategory of $\operatorname{Gr}(k[T])$ and $\Sigma \cM = \cM.$ Hence  $F^{-1}\cM$ is a triangulated subcategory of $\sD(R)$, by Lemma \ref{lem:ex34a}. 
\end{proof}

\noindent The following is a well-known property of the polynomial algebra in one variable.

\begin{Lemma} \label{lem:nts}
We have that 
$$\operatorname{gr}(k[T]) = \add\{\Sigma^i k[T], \Sigma^j k[T]/(T^{n+1}) \hspace{1mm} | \hspace{1mm} i,j \in \BZ, n \in \BN_0\},$$ and hence 
$$\cM = \add\{(\Sigma^i k[T])^{\ast}, (\Sigma^j k[T]/(T^{n+1}))^{\ast} \hspace{1mm} | \hspace{1mm} i,j \in \BZ, n \in \BN_0\}.$$
\end{Lemma}

\begin{Lemma} \label{lem:functoranddual1}
We have the following isomorphisms. 
{} \\
(1) $F(\Sigma^i X_n) \cong \Sigma^i k[T]/(T^{n+1}).$ \\
(2) $F(\sE_m) \cong \Sigma^m k[T^{-1}].$
\end{Lemma}

\begin{proof} \leavevmode \\
(1) Because $F\Sigma \cong \Sigma F,$ it is enough to show that $F(X_n) \cong k[T]/(T^{n+1})$. By \citep[Remark 1.3]{J}, we have a distinguished triangle 
\begin{equation} \label{equ:inddistri}
\xymatrix{
\Sigma^{n+1}R \ar[rr]^-{\cdot T^{n+1}} && R \ar[rr] && X_n
}
\end{equation}
and by Lemma \ref{lem:Fishom}, we obtain the following long exact sequence after applying $F.$
$$ \xymatrix{
\cdots \ar[r] & F(\Sigma^{n+1} R) \ar[rr]^-{F(\cdot T^{n+1})} && F(R) \ar[rr] && F(X_n) \ar[r] & \cdots 
}
$$
This becomes
$$\xymatrix{
\cdots \ar[r] & \Sigma^{n+1} A \ar[rr]^-{\cdot T^{n+1}} && A \ar[r] & F(X_n) \ar[r] & \Sigma^{n+2}A \ar[rr]^-{\Sigma(\cdot T^{n+1})} && \Sigma A \ar[r] & \cdots
}$$
where $A = (R,\Sigma^{\ast}R) \cong k[T].$ Now, $\cdot T^{n+1} : \Sigma^{n+1} A \rightarrow A$ is injective, whence $F(X_n) \rightarrow \Sigma^{n+2}A$ is the zero map because the sequence is exact and $\Sigma(\cdot T^{n+1}) : \Sigma^{n+2}A \rightarrow \Sigma A$ is injective. Therefore, $A \rightarrow F(X_n)$ is surjective, and by the first isomorphism theorem, $F(X_n) \cong A/\Image(\cdot T^{n+1}) \cong k[T]/(T^{n+1}).$
\leavevmode \\
\leavevmode \\
(2) Because $F(\sE_m) = F(\Sigma^m \sE_0) \cong \Sigma^m F(\sE_0),$ it is enough to show that $F(\sE_0) \cong k[T^{-1}].$ Now, 
$$\hocolim(\Sigma^0 X_0 \rightarrow \Sigma^{-1}X_1 \rightarrow \cdots)=$$ 
$$\hocolim(\xymatrix@1{R/(T) \ar[r]^-{\cdot T} & \Sigma^{-1}R/(T^2) \ar[r]^-{\cdot T} & \cdots
})=\sE_0.
$$
Now, $R$ is a compact object of the derived category $\sD(R)$, so by Lemma 2.8 of \citep{N}, applying the functor $F$ gives
$$F(\sE_0) = F(\hocolim(\xymatrix@1{R/(T) \ar[r]^-{\cdot T} & \Sigma^{-1}R/(T^2) \ar[r]^-{\cdot T} & \cdots
})) = 
$$
$$
\colim (F(\xymatrix@1{R/(T) \ar[r]^-{\cdot T} & \Sigma^{-1}R/(T^2) \ar[r]^-{\cdot T} & \cdots
})) =
$$
$$
\colim (\xymatrix@1{A/(T) \ar[r]^-{\cdot T} & \Sigma^{-1}A/(T^2) \ar[r]^-{\cdot T} & \cdots
}) = k[T^{-1}].$$
\end{proof}

\begin{Lemma} \label{lem:functoranddual2}
We have the following isomorphisms in $\operatorname{Gr}(k[T])$. 
{} \\
(A) $(\Sigma^{-i-n} k[T]/(T^{n+1}))^{\ast} \cong \Sigma^i k[T]/(T^{n+1}).$ \\
(B) $(\Sigma^{-m} k[T])^{\ast} \cong \Sigma^m k[T^{-1}].$
\end{Lemma}

\begin{proof} \leavevmode \\
(A) Because $(\Sigma^i M)^{\ast} \cong \Sigma^{-i} (M^{\ast}),$ it is enough to show that $(k[T]/(T^{n+1}))^{\ast} \cong \Sigma^{-n}k[T]/(T^{n+1}).$ Consider the following component of an $A$-module homomorphism,
$$\Psi^i : \Hom \Big( (k[T]/(T^{n+1}))^{-i},k \Big) \rightarrow \Big( k[T]/(T^{n+1}) \Big) ^{i-n}$$
defined on $0 \leq i \leq n$ by $\Psi^i(\alpha) = \alpha(T^i) \cdot T^{n-i}$ and by $\Psi^i(\alpha) = 0$ if $i$ is not in this range. To see that this is indeed compatible with multiplication by $A$-elements, consider the following diagram.
\[ 
  \xymatrix {
	\Hom \Big( (k[T]/(T^{n+1}))^{-i+1},k \Big) \ar[dd]^-{\Psi^{i-1}} && \Hom \Big( (k[T]/(T^{n+1}))^{-i},k \Big) \ar[ll]_-{\cdot T} \ar[dd]^-{\Psi^i} \\ \\
	\Big( k[T]/(T^{n+1}) \Big) ^{i-n-1} && \Big( k[T]/(T^{n+1}) \Big) ^{i-n} \ar[ll]^-{\cdot T}
 	}
\]
If $0 \leq i-1 < i \leq n$ and $\alpha \in \Hom \Big( (k[T]/(T^{n+1}))^{-i},k \Big),$ then diagram-chasing yields 
$$T \cdot \Psi^i(\alpha) = T\alpha(T^i)\cdot T^{n-i} = \alpha(T^i)\cdot T^{1+n-i} = \alpha(T^{i-1+1})\cdot T^{1+n-i} =$$ 
$$(\alpha T)(T^{i-1})\cdot T^{1+n-i} = \Psi^{i-1}(\alpha T).$$ 
If $i=0$, then $T \cdot \Psi^0(\alpha) = T\alpha(T^0)\cdot T^{n} = \alpha(T^0) \cdot T^{n+1} = 0 = \Psi^{-1}(\alpha T)$ as required. Finally at each $i$ we have a $k$-linear map which is nonzero for $0 \leq i \leq n$ between $1$-dimensional $k$-vector spaces and hence each $\Psi^i$ is an isomorphism of $k$-vector spaces (and for $i$ outside $0,\dots,n,$ it is just the zero map). This, combined with compatibility with $A$-multiplication gives the result, namely that $(k[T]/(T^{n+1}))^{\ast} \cong \Sigma^{-n}k[T]/(T^{n+1}).$
\leavevmode \\
\leavevmode \\
(B) Because $(\Sigma^i M)^{\ast} \cong \Sigma^{-i} (M^{\ast}),$ it is enough to show that $(k[T])^{\ast} \cong k[T^{-1}].$ Consider the following component of an $A$-module homomorphism,
$$\Phi^i : \Hom \Big( (k[T])^{-i},k \Big) \rightarrow \Big( k[T^{-1}] \Big) ^{i} : \alpha \mapsto \alpha(T^i)\cdot T^{-i}.$$
Again, we check compatibility with multiplication by $A$-elements. Consider the following diagram.
\[ 
  \xymatrix {
	\Hom \Big( (k[T])^{-i+1},k \Big) \ar[dd]^-{\Phi^{i-1}} && \Hom \Big( (k[T])^{-i},k \Big) \ar[ll]_-{\cdot T} \ar[dd]^-{\Phi^i} \\ \\
	\Big( k[T^{-1}] \Big) ^{i-1} && \Big( k[T^{-1}] \Big) ^{i} \ar[ll]^-{\cdot T}
 	}
\]
Diagram-chasing is then even simpler than in part (A), since $\alpha \in \Hom \Big( (k[T])^{-i},k \Big)$ can easily be seen to map to $\alpha(T^i)\cdot T^{1-i} \in \Big( k[T^{-1}] \Big) ^{i-1}$ after following either direction in the diagram.  For similar reasons as in part (A), we have that the collection $(\Phi^i)_{i \in \BZ}$ is an isomorphism between $(k[T])^{\ast}$ and $k[T^{-1}].$
\end{proof}

\begin{Corollary} \label{cor:ntsplusfunctoranddual}
By Lemmas \ref{lem:nts},\ref{lem:functoranddual1} and \ref{lem:functoranddual2}, 
$$\cM = \add\{F(\Sigma^i X_n), F(\sE_m) \hspace{1mm} | \hspace{1mm} i,m \in \BZ, n \in \BN_0 \}.$$
\end{Corollary}

\begin{Lemma} \label{lem:KYZ}
We have that $\Dbar = F^{-1}\cM.$
\end{Lemma}

\begin{proof}
Corollary \ref{cor:ntsplusfunctoranddual} implies 
$$F(\Dbar) = F(\add\{\Sigma^iX_n,\sE_m \hspace{1.5mm} | \hspace{1.5mm} i,m \in \BZ, n \in \BN_0\})\subseteq$$
$$\add\{F(\Sigma^iX_n),F(\sE_m) \hspace{1.5mm} | \hspace{1.5mm} i,m \in \BZ, n \in \BN_0\} =\cM.$$ 
On the other hand, Lemmas \ref{lem:nts} and \ref{lem:functoranddual2} show that each object in $\cM$ is a direct sum of finitely many objects from $\{\Sigma^ik[X]/(X^{n+1}),\Sigma^mk[T^{-1}] \hspace{1.5mm} | \hspace{1.5mm}  i,m \in \BZ, n \in \BN_0\},$ so $\cM \subseteq F(\add\{\Sigma^iX_n,\sE_m \hspace{1.5mm} | \hspace{1.5mm} i,m \in \BZ, n \in \BN_0\}) =$ $F(\Dbar)$ by Lemma \ref{lem:functoranddual1}. We conclude $F(\Dbar) = \cM$ which implies $\Dbar = F^{-1}\cM$ by \citep[thm 3.1]{KYZ}.
\end{proof}

\begin{Theorem} \label{cor:Distri}
The category $\Dbar$ is a triangulated subcategory of $\sD(R).$
\end{Theorem}

\begin{proof}
By Lemmas \ref{lem:ex34} and \ref{lem:KYZ}, the result follows.
\end{proof}

\section{Morphisms in the category $\Dbar$}
\label{sec:morphisms}

In this section we compute homs between finite objects (i.e. the $X_n$ and their shifts) and hocolimits (i.e. the $\sE_m$ and their shifts) in the category $\Dbar$ and also homs between hocolimits.

\begin{Definition} \label{def:wedge}
Let $X \in \ind\cD,$ i.e., $X$ is an indecomposable object of $\cD,$ which means that $X$ is represented by a vertex in the Auslander-Reiten quiver of $\cD$. We write $X$ uniquely as $\Sigma^{n-i}X_i$ for $n \in \BZ$, $i \in \BN_0$. Then there is a unique slice associated with the direct system
$$ \Sigma^nX_0 \rightarrow \Sigma^{n-1}X_1 \rightarrow \cdots \rightarrow \Sigma^{n-i}X_i \rightarrow \cdots$$
containing $\Sigma^{n-i}X_i$. We define the \textit{wedge associated with} $\Sigma^{n-i}X_i$ (or \textit{wedge based at} $\Sigma^{n-i}X_i$ if $i=0$) as
$$\BW(\Sigma^{n-i}X_i) = \{ \Sigma^{n-j}X_k : 0 \leq j \leq k \}.$$
Pictorially, we have the following, where the subset contains the edges.
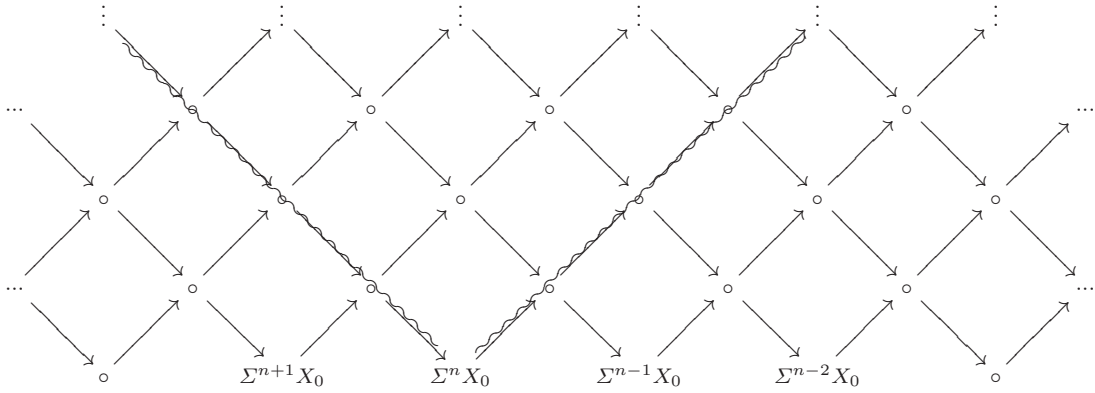
\[ \def\objectstyle{\scriptstyle}
  \xymatrix @-2.5pc @! {
    & \vdots \ar[dr] & & \vdots \ar[dr] & & \vdots \ar[dr] & & \vdots \ar[dr] & & \vdots \ar[dr] & & \vdots & \\
    \cdots \ar[dr]& & \circ \ar[ur] \ar[dr] & & \circ \ar[ur] \ar[dr] & & \circ \ar[ur] \ar[dr] & & \circ \ar[ur] \ar[dr] & & \circ \ar[ur] \ar[dr] & & \cdots \\
    & \circ \ar[ur] \ar[dr] & & \circ \ar[ur] \ar[dr] & & \circ \ar[ur] \ar[dr] & & \circ \ar[ur] \ar[dr] & & \circ \ar[ur] \ar[dr] & & \circ \ar[ur] \ar[dr] & \\
    \cdots \ar[ur]\ar[dr]& & \circ \ar[ur] \ar[dr] & & \circ \ar[ur] \ar[dr] & & \circ \ar[ur] \ar[dr] & & \circ \ar[ur] \ar[dr] & & \circ \ar[ur] \ar[dr] & & \cdots\\
    & \circ \ar[ur] & & \Sigma^{n+1} X_0 \ar[ur] & & \Sigma^n X_0 \ar[ur] \ar@{~}[uuuullll] \ar@{~}[uuuurrrr] & & \Sigma^{n-1}X_0 \ar[ur] & & \Sigma^{n-2}X_0 \ar[ur] & & \circ \ar[ur] & \\
               }
\]
Notice that $\BW(\Sigma^{n-i}X_i)$ is independent of the number $i$. In particular, a general indecomposable object $\Sigma^{n-i}X_i$ has wedge based at $\Sigma^nX_0$. Therefore, we may choose to write $\BW(\Sigma^{n}X_0)$ instead of $\BW(\Sigma^{n-i}X_i)$, because they are equal. \demo
\end{Definition}

\begin{Notation} \label{not:Hplusminus}
The sets
$$H^{-}(\Sigma^r X_s) = \{\Sigma^{-n}X_{n-m-2} \hspace{1.5mm} | \hspace{1.5mm} m \leq -r-s-3, -r-s-1 \leq n \leq -r-1 \},$$ 
$$H^{+}(\Sigma^r X_s) = \{\Sigma^{-n}X_{n-m-2} \hspace{1.5mm} | \hspace{1.5mm} -r-s-1 \leq m \leq -r-1, -r+1 \leq n \}$$
and $H(\Sigma^r X_s)=H^{-}(\Sigma^r X_s) \cup H^{+}(\Sigma^r X_s)$ were originally defined in \citep[Definition 2.1]{J}. These subsets are illustated below.
$$
\vcenter{
  \xymatrix @-5.1pc @! {
    &&&*{} &&&&&&&& *{}&& \\
    &&&& *{} \ar@{--}[ul] & & & & & & *{} \ar@{--}[ur] \\
    &*{}&& H^-(\Sigma^r X_s) & & & & & & & & H^+(\Sigma^r X_s) && *{}\\
    &&*{}\ar@{--}[ul]& & & & {\scriptscriptstyle \Sigma^{r+1} X_s\hspace{3ex}} \ar@{-}[ddll] \ar@{-}[uull] & {\scriptscriptstyle \Sigma^r X_s} & {\hspace{3ex}\scriptscriptstyle \Sigma^{r-1} X_s} \ar@{-}[ddrr] \ar@{-}[uurr]& & &&*{}\ar@{--}[ur]&\\ 
    && \\
    *{}\ar@{--}[r]&*{} \ar@{-}[rrr] && & {\scriptscriptstyle \Sigma^{r+s+1} X_0} \ar@{-}[uull]\ar@{-}[rrrrrr]& & & & & & {\scriptscriptstyle \Sigma^{r-1} X_0} \ar@{-}[uurr]\ar@{-}[rrr]&&&*{}\ar@{--}[r]&*{}\\
           }
}
$$
In this picture, the subsets include the edges. \demo

\end{Notation}

\begin{Lemma}\label{lem:lemmadual}
Let
\begin{equation}\label{equ:slice}
\xymatrix{
\Sigma^nX_0 \ar[r]^-{\xi_0} & \Sigma^{n-1}X_1 \ar[r]^-{\xi_1} & \Sigma^{n-2}X_2 \ar[r]& \cdots
}
\end{equation}
be a direct system associated with a slice in the Auslander-Reiten quiver of $\cD$, and let $Y$ be any indecomposable object of $\cD$. Consider the direct system
\begin{equation} \label{equ:origchain}
\xymatrix{
(Y,\Sigma^nX_0) \ar[r]^-{(Y,\xi_0)} & (Y,\Sigma^{n-1}X_1) \ar[r]^-{(Y,\xi_1)} & (Y,\Sigma^{n-2}X_2) \ar[r]& \cdots,
}
\end{equation} which is obtained by applying the functor $(Y,-)$ to the direct system (\ref{equ:slice}) above. Then the direct system (\ref{equ:origchain}) is naturally isomorphic to the direct system \begin{equation} \label{equ:newchain}
\xymatrix{
(\Sigma^nX_0,\Sigma^2Y)^\ast \ar[rr]^-{(\xi_0,\Sigma^2Y)^\ast} &&   (\Sigma^{n-1}X_1,\Sigma^2Y)^\ast \ar[rr]^-{(\xi_1,\Sigma^2Y)^\ast} &&  (\Sigma^{n-2}X_2,\Sigma^2Y)^\ast \ar[r] &  \cdots.
}
\end{equation}
Furthermore,
\begin{equation}\label{equ:kor0} (\Sigma^{n-i}X_{i},\Sigma^2Y)^\ast \cong \Big \{ \begin{matrix} k \hspace{1.5 mm} \textrm{if} \hspace{2 mm} \Sigma^2Y \in H(\Sigma^{n-i+1}X_i), \\ 0 \hspace{1.5 mm} \textrm{if} \hspace{2 mm} \Sigma^2Y \not\in H(\Sigma^{n-i+1}X_i). \end{matrix} \end{equation}
\end{Lemma}

\begin{proof}
Consider an arbitrary map $\xymatrix{(Y,\Sigma^{n-i}X_i) \ar[r]^-{(Y,\xi_i)} & (Y,\Sigma^{n-(i+1)}X_{i+1})}$ in the direct system (\ref{equ:origchain}). Then this is a morphism of vector spaces over $k$, and there is a commutative diagram
\begin{equation} \begin{aligned} \label{equ:natism}
\xymatrix{ (Y,\Sigma^{n-i}X_i) \ar[rr]^-{(Y,\xi_i)} \ar@{->}[d]^-{\cong} & & (Y,\Sigma^{n-(i+1)}X_{i+1}) \ar@{->}[d]^-{\cong} \\
(Y,\Sigma^{n-i}X_i)^{\ast \ast} \ar[rr]^-{(Y,\xi_i)^{\ast \ast}} && (Y,\Sigma^{n-(i+1)}X_{i+1})^{\ast \ast} }
\end{aligned} \end{equation}
where $^\ast$ is the contravariant functor from the category of $k$-vector spaces to itself, $^\ast : {\operatorname{\textsf Vect}}_k \rightarrow {\operatorname{\textsf Vect}}_k$, which maps a vector space over $k$ to its dual space and maps a linear map to its transpose. We can rewrite the bottom morphism of (\ref{equ:natism}) as
\begin{equation} \label{equ:dual} \xymatrix{ ((Y,\Sigma^{n-i}X_i)^\ast)^\ast \ ( && )^\ast \ ((Y,\Sigma^{n-(i+1)}X_{i+1})^{\ast})^\ast \ar[ll]_-{(Y,\xi_i)^{\ast}} }.
\end{equation}
By \citep[Remark 1.2]{J}, the category $\cD$ is 2-Calabi-Yau and hence Serre duality states that $(Y,\Sigma^{n-i}X_i) \cong (\Sigma^{n-i}X_i,\Sigma^2 Y)^\ast.$ Now, functors preserve isomorphisms, so we can apply $^\ast$ to this to obtain $(Y,\Sigma^{n-i}X_i)^\ast \cong (\Sigma^{n-i}X_i,\Sigma^2 Y)^{\ast \ast}$, which is naturally isomorphic to $(\Sigma^{n-i}X_i,\Sigma^2 Y)$. Therefore, (\ref{equ:dual}) is naturally isomorphic to
\begin{equation} \label{equ:dual2} \xymatrix{ (\Sigma^{n-i}X_i,\Sigma^2 Y)^\ast \ar[rr]^-{(\xi_i,\Sigma^2 Y)^{\ast}} && (\Sigma^{n-(i+1)}X_{i+1},\Sigma^2 Y)^\ast  }.
\end{equation}
Combining (\ref{equ:dual2}) with (\ref{equ:natism}) shows that
$$\xymatrix{(Y,\Sigma^{n-i}X_i) \ar[r]^-{(Y,\xi_i)} & (Y,\Sigma^{n-(i+1)}X_{i+1})}$$
is naturally isomorphic to
$$\xymatrix{ (\Sigma^{n-i}X_i,\Sigma^2 Y)^\ast \ar[rr]^-{(\xi_i,\Sigma^2 Y)^{\ast}} && (\Sigma^{n-(i+1)}X_{i+1},\Sigma^2 Y)^\ast  },$$
and hence the direct system (\ref{equ:origchain}) is naturally isomorphic to the direct system (\ref{equ:newchain})
$$
\xymatrix{
(\Sigma^nX_0,\Sigma^2Y)^\ast \ar[rr]^-{(\xi_0,\Sigma^2Y)^\ast} &&   (\Sigma^{n-1}X_1,\Sigma^2Y)^\ast \ar[rr]^-{(\xi_1,\Sigma^2Y)^\ast} &&  (\Sigma^{n-2}X_2,\Sigma^2Y)^\ast \ar[r] &  \cdots,
}
$$
as required. We now prove claim (\ref{equ:kor0}) of the lemma. By \citep[Proposition 2.2]{J}, we have that
$$(\Sigma^{n-i}X_{i},\Sigma^2Y) \cong \Big \{ \begin{matrix} k \hspace{1.5 mm} \textrm{if} \hspace{1.5 mm} \Sigma^2Y \in H(\Sigma^{n-i+1}X_i), \\ 0 \hspace{1.5 mm} \textrm{if} \hspace{1.5 mm} \Sigma^2Y \not\in H(\Sigma^{n-i+1}X_i). \end{matrix} $$
This directly implies that
$$ (\Sigma^{n-i}X_{i},\Sigma^2Y)^\ast \cong \Big \{ \begin{matrix} k \hspace{1.5 mm} \textrm{if} \hspace{1.5 mm} \Sigma^2Y \in H(\Sigma^{n-i+1}X_i), \\ 0 \hspace{1.5 mm} \textrm{if} \hspace{1.5 mm} \Sigma^2Y \not\in H(\Sigma^{n-i+1}X_i), \end{matrix} $$
because a vector space which is finite-dimensional has the same dimension as its dual.
\end{proof}

\begin{Lemma}\label{lem:colim}
The colimit of the direct system (\ref{equ:origchain}),
$$
\xymatrix{
(Y,\Sigma^nX_0) \ar[r]^-{(Y,\xi_0)} & (Y,\Sigma^{n-1}X_1) \ar[r]^-{(Y,\xi_1)} & (Y,\Sigma^{n-2}X_2) \ar[r]& \cdots,
}
$$
is isomorphic to $k$ if $Y \in \BW(\Sigma^nX_0),$ and $0$ if $Y \not\in \BW(\Sigma^nX_0).$
\end{Lemma}

\begin{proof}
Using Lemma \ref{lem:lemmadual}, we can rewrite the direct system as
$$
\xymatrix{
(\Sigma^nX_0,\Sigma^2Y)^\ast \ar[rr]^-{(\xi_0,\Sigma^2Y)^\ast} &&   (\Sigma^{n-1}X_1,\Sigma^2Y)^\ast \ar[rr]^-{(\xi_1,\Sigma^2Y)^\ast} &&  (\Sigma^{n-2}X_2,\Sigma^2Y)^\ast \ar[r] &  \cdots,
}
$$
from Equation (\ref{equ:newchain}). Now consider the three numbered regions in the Auslander-Reiten quiver below.
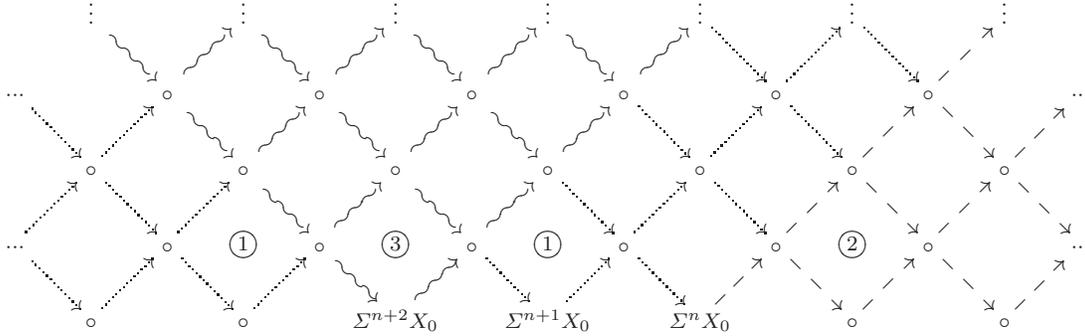
\[ \def\objectstyle{\scriptstyle}
  \xymatrix @-2.9pc @! {
    & \vdots \ar@{~>}[dr] & & \vdots \ar@{~>}[dr] & & \vdots \ar@{~>}[dr] & & \vdots \ar@{~>}[dr] & & \vdots \ar@{.>}[dr] & & \vdots \ar@{.>}[dr] & & \vdots & \\
    \cdots \ar@{.>}[dr]& & \circ \ar@{~>}[ur] \ar@{~>}[dr] & & \circ \ar@{~>}[ur] \ar@{~>}[dr] &  & \circ \ar@{~>}[ur] \ar@{~>}[dr] & & \circ \ar@{~>}[ur] \ar@{.>}[dr] & & \circ \ar@{.>}[ur] \ar@{.>}[dr] & & \circ \ar@{-->}[ur] \ar@{-->}[dr] & & \cdots \\
    & \circ \ar@{.>}[ur] \ar@{.>}[dr] &  & \circ \ar@{~>}[ur] \ar@{~>}[dr] &  & \circ \ar@{~>}[ur] \ar@{~>}[dr] & & \circ \ar@{~>}[ur] \ar@{.>}[dr] &  & \circ \ar@{.>}[ur] \ar@{.>}[dr] &  & \circ \ar@{-->}[ur] \ar@{-->}[dr] &  & \circ \ar@{-->}[ur] \ar@{-->}[dr] & \\
    \cdots \ar@{.>}[ur]\ar@{.>}[dr]& & \circ \ar@{.>}[ur] \ar@{.>}[dr] & *+[o][F]{1} & \circ \ar@{~>}[ur] \ar@{~>}[dr] & *+[o][F]{3} & \circ \ar@{~>}[ur] \ar@{.>}[dr] & *+[o][F]{1} & \circ \ar@{.>}[ur] \ar@{.>}[dr] & & \circ \ar@{-->}[ur] \ar@{-->}[dr] & *+[o][F]{2} & \circ \ar@{-->}[ur] \ar@{-->}[dr] & & \cdots\\
    & \circ \ar@{.>}[ur] &  & \circ \ar@{.>}[ur] & & \Sigma^{n+2}X_0 \ar@{~>}[ur] & & \Sigma^{n+1}X_0 \ar@{.>}[ur] & & \Sigma^{n}X_0 \ar@{-->}[ur] & & \circ \ar@{-->}[ur] & & \circ \ar@{-->}[ur] & \\
               }
\]
Region $\def\objectstyle{\scriptstyle}\xymatrix{*+[o][F]{3}}$ is highlighted with wavy arrows and contains all of the vertices on its boundary, so, in particular, equals $\BW(\Sigma^{n+2}X_0).$ Region $\def\objectstyle{\scriptstyle}\xymatrix{*+[o][F]{2}}$ is highlighted with dashed arrows and contains all of the vertices on its boundary. Region $\def\objectstyle{\scriptstyle}\xymatrix{*+[o][F]{1}}$ contains none of the vertices which have wavy or dashed arrows coming in or going out.
Recall the following useful fact from Lemma \ref{lem:lemmadual}.
$$(\Sigma^{n-i}X_{i},\Sigma^2Y)^\ast \cong \Big \{ \begin{matrix} k \hspace{1.5 mm} \textrm{if} \hspace{1.5 mm} \Sigma^2Y \in H(\Sigma^{n-i+1}X_i), \\ 0 \hspace{1.5 mm} \textrm{if} \hspace{1.5 mm} \Sigma^2Y \not\in H(\Sigma^{n-i+1}X_i). \end{matrix}$$
If the indecomposable object $\sig Y$ is located in region $\def\objectstyle{\scriptstyle}\xymatrix{*+[o][F]{1}}$, then $\sig Y$ is never in $H(\Sigma^{n-i+1}X_i)$, for any $i \in \BN_0$. If $\sig Y$ is located in region $\def\objectstyle{\scriptstyle}\xymatrix{*+[o][F]{2}}$, then $\sig Y$ will be in $H(\Sigma^{n-i+1}X_i)$ for only finitely many values of $i$. This is because, once $i$ is large enough, $\sig Y$ will be to the left of $H^{+}(\Sigma^{n-i+1}X_i)$. Therefore, if $\sig Y$ is located in regions $\def\objectstyle{\scriptstyle}\xymatrix{*+[o][F]{1}}$ or $\def\objectstyle{\scriptstyle}\xymatrix{*+[o][F]{2}}$, the categorical colimit of the direct system (\ref{equ:newchain}) is trivially zero. What remains is to check what happens when $\sig Y$ is in region $\def\objectstyle{\scriptstyle}\xymatrix{*+[o][F]{3}}$. Note that this means $\sig Y \in \BW(\Sigma^{n+2}X_0),$ or equivalently, $Y \in \BW(\Sigma^{n}X_0).$

Suppose $\sig Y$ is in region $\def\objectstyle{\scriptstyle}\xymatrix{*+[o][F]{3}}$. Then there exists $N \in \BN_0$ such that whenever $i\geq N$, $\sig Y \in H^{-}(\Sigma^{n-i+1}X_i) \subset H(\Sigma^{n-i+1}X_i)$. Hence, the direct system (\ref{equ:newchain}) looks like
$$ 0 \rightarrow 0 \rightarrow \dots \rightarrow 0 \rightarrow k \rightarrow k \rightarrow \cdots$$
with at most $N$ zeroes. We claim that the colimit of this system is $k$. Consider one of the maps in the direct system (\ref{equ:newchain}) between two of the nonzero hom spaces. Denote this map
\begin{equation}\label{equ:nonzeromap}
\xymatrix{
(\Sigma^{n-i}X_i,\Sigma^2Y)^\ast \ar[rr]^-{(\xi_i,\Sigma^2Y)^\ast}  &&  (\Sigma^{n-(i+1)}X_{i+1},\Sigma^2Y)^\ast
}
\end{equation}
for $i \geq N$, so that $\sig Y \in H^{-}(\Sigma^{n-i+1}X_i).$ It is enough to show that such a map is nonzero, because if that is the case, then it is an isomorphism; this is due to the fact that the hom spaces are one-dimensional. Observe that $\sig Y \in H^{-}(\Sigma^{n-i+1}X_i)$ is equivalent to requiring that 
\begin{equation}\label{equ:observation} Y \in H^{-}(\Sigma^{n-(i+1)}X_i). \end{equation}
By Lemma \ref{lem:lemmadual}, (\ref{equ:nonzeromap}) is naturally isomorphic to 
\begin{equation}\label{equ:nonzeromapism}
\xymatrix{ (Y,\Sigma^{n-i}X_i) \ar[rr]^-{(Y,\xi_i)} && (Y,\Sigma^{n-(i+1)}X_{i+1}).}
\end{equation}
Let $\varphi \in (Y,\Sigma^{n-i}X_i)$ be nonzero. Then the image of $\varphi$ under $(Y,\xi_i)$ is nonzero, by virtue of Lemma 2.5 of \citep{J}. This is because $Y,$ $\Sigma^{n-i}X_i$ and $\Sigma^{n-i-1}X_{i+1}$ are indecomposable objects of $\cD$  such that $\Sigma^{n-i}X_i, \Sigma^{n-i-1}X_{i+1} \in H^{+}(\Sigma Y)$ and $\Sigma^{n-i-1}X_{i+1} \in H^{+}(\Sigma^{n-i+1}X_i)$, as shown in the following sketch.
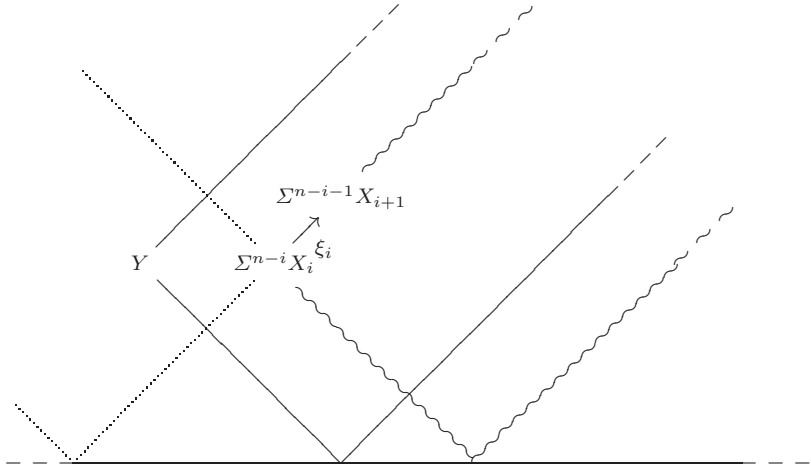
\[ \def\objectstyle{\scriptstyle}
  \xymatrix @-4.6pc @! {
    & & & & & & & & \\
    & & & & & *{} \ar@{--}[ur] &  & *{} \ar@{~~}[ur] & *{} \\
    & & & & & & & & & & \\
    & & & & & \Sigma^{n-i-1}X_{i+1} \ar@{~}[uurr] & &  &  & *{}\ar@{--}[ur] & & & \\
    & & Y \ar@{-}[dddrrr] \ar@{-}[uuurrr] & & \Sigma^{n-i}X_{i} \ar@{..}[dddlll] \ar@{..}[uuulll] \ar@{~}[dddrrr] \ar@{->}[ur]_-{\xi_i} & & & & & *{} & *{}\ar@{~~}[ur]\\
    *{} & & & & &&&&&\\ 
    & & & & & & & *{} \\
    *{} \ar@{--}[r] & *{} \ar@{..}[ul] \ar@{-}[rrr] & & \ar@{-}[rrrrrrrr] & & *{} \ar@{-}[uuuurrrr] & &*{}\ar@{~}[uuurrr]&&&&*{}\ar@{--}[r]&*{}\\
           }
\]
The region $H^{-}(\Sigma^{n-(i+1)}X_i)$ is illustrated by dotted lines; this is where $Y$ is located. Notice that wherever $Y$ is in this region, $H^{+}(\Sigma Y)$ contains $\xi_i : \Sigma^{n-i}X_i \rightarrow \Sigma^{n-i-1}X_{i+1}$. The wavy lines illustrate the region $H^{+}(\Sigma^{n-i+1}X_i)$. Lemma 2.5 of \citep{J} tells us that in this situation, the composition of nonzero morphisms $Y \rightarrow \Sigma^{n-i}X_i$ and $\Sigma^{n-i}X_i \rightarrow \Sigma^{n-i-1}X_{i+1}$ is nonzero. We have therefore shown that the morphism (\ref{equ:nonzeromap}) is nonzero when $\Sigma^2 Y \in \BW(\Sigma^{n+2}X_0)$ and $i \geq N,$ so that $\sig Y \in H^{-}(\Sigma^{n-i+1}X_i)$ holds. Hence, the morphisms between nonzero hom spaces in the direct system (\ref{equ:newchain}) are isomorphisms, and the colimit of this direct system is isomorphic to $k$ whenever $\Sigma^2 Y$ is in region $\def\objectstyle{\scriptstyle}\xymatrix{*+[o][F]{3}}$. Now, the direct system (\ref{equ:newchain}),
$$\xymatrix{
(\Sigma^nX_0,\Sigma^2Y)^\ast \ar[rr]^-{(\xi_0,\Sigma^2Y)^\ast} &&   (\Sigma^{n-1}X_1,\Sigma^2Y)^\ast \ar[rr]^-{(\xi_1,\Sigma^2Y)^\ast} &&  (\Sigma^{n-2}X_2,\Sigma^2Y)^\ast \ar[r] &  \cdots,
}$$
and the direct system (\ref{equ:origchain}),
$$
\xymatrix{
(Y,\Sigma^nX_0) \ar[r]^-{(Y,\xi_0)} & (Y,\Sigma^{n-1}X_1) \ar[r]^-{(Y,\xi_1)} & (Y,\Sigma^{n-2}X_2) \ar[r]& \cdots,
}
$$
are naturally isomorphic to one another, and so the colimit of (\ref{equ:origchain}) is isomorphic to $k$ if $Y \in \BW(\Sigma^nX_0),$ and $0$ if $Y \not\in \BW(\Sigma^nX_0),$ as required.
\end{proof}

\begin{Proposition}\label{pro:homsin}
Consider the direct system
$$ \Sigma^nX_0 \rightarrow \Sigma^{n-1}X_1 \rightarrow \cdots \rightarrow \Sigma^{n-i}X_i \rightarrow \cdots$$
associated with a slice in the Auslander-Reiten quiver of $\cD$. If $Y \in \ind\cD$ is any indecomposable object, then
$$\hmD(Y,\hocolim_i(\Sigma^{n-i}X_i)) \cong \Big \{ \begin{matrix} k \hspace{1.5 mm} \textrm{if} \hspace{2 mm} Y \in \BW(\Sigma^{n}X_0), \\ 0 \hspace{1.5 mm} \textrm{if} \hspace{2 mm} Y \not\in \BW(\Sigma^{n}X_0). \end{matrix} $$
\end{Proposition}

\begin{proof}
The indecomposable objects of $\cD$ are compact in $\sD(R)$. This is because $R$ and its associated shifts are compact, and the distinguished triangle (\ref{equ:inddistri}) gives that $X_n$ is compact. Hence by \citep[Lemma 2.8]{N}, we have that
\begin{equation}\label{hocolimcong}(Y, \hocolim_i(\Sigma^{n-i}X_i)) \cong \colim_i(Y, \Sigma^{n-i}X_i).\end{equation}
This is isomorphic to $k$ if $Y \in \BW(\Sigma^nX_0),$ and $0$ if $Y \not\in \BW(\Sigma^nX_0)$ by Lemma \ref{lem:colim}.
\end{proof}

\begin{Corollary} \label{cor:homsout}
Consider the direct system
$$ \Sigma^nX_0 \rightarrow \Sigma^{n-1}X_1 \rightarrow \cdots \rightarrow \Sigma^{n-i}X_i \rightarrow \cdots$$
associated with a slice in the Auslander-Reiten quiver of $\cD$. If $Y \in \ind\cD$ is any indecomposable object, then
$$\hmD(\hocolim_i(\Sigma^{n-i}X_i),Y) \cong \Big \{ \begin{matrix} k \hspace{1.5 mm} \textrm{if} \hspace{2 mm} Y \in \BW(\Sigma^{n+2}X_0), \\ 0 \hspace{1.5 mm} \textrm{if} \hspace{2 mm} Y \not\in \BW(\Sigma^{n+2}X_0). \end{matrix} $$
\end{Corollary}

\begin{proof}
There is a short exact sequence
$$0 \rightarrow \textrm{lim}^1_i (\Sigma^{1+n-i}X_i,\sig Y) \rightarrow (\hocolim_i(\Sigma^{n-i}X_i),\sig Y) \rightarrow \textrm{lim}_i(\Sigma^{n-i}X_i,\sig Y) \rightarrow 0$$
by \citep[Lemma 1.13.1]{JSM}, where $\textrm{lim}^1$ is the first right-derived functor of $\operatorname{lim}$. Now, for each $i \in \BN_0,$ the space $(\Sigma^{1+n-i}X_i,\sig Y)$ is either isomorphic to $0$ or $k$, and therefore 
$${\lim}{}^1 _i(\Sigma^{1+n-i}X_i,\Sigma^2 Y) = 0,$$
by Exercise 3.5.2 of \citep{W}. Therefore,
\begin{equation}\label{equ:limiso}
(\hocolim_i(\Sigma^{n-i}X_i),\sig Y) \cong \textrm{lim}_i(\Sigma^{n-i}X_i,\sig Y).
\end{equation}

Now, apply the functor $(-,\Sigma^2 Y)$ to the direct system
$$
\xymatrix{
\Sigma^nX_0 \ar[r]^-{\xi_0} & \Sigma^{n-1}X_1 \ar[r]^-{\xi_1} & \Sigma^{n-2}X_2 \ar[r]& \cdots
}
$$
to obtain the inverse system
\begin{equation}\label{equ:inversechain}
\xymatrix{
(\Sigma^nX_0,\sig Y) && (\Sigma^{n-1}X_1,\sig Y) \ar[ll]_-{(\xi_0,\sig Y)} && (\Sigma^{n-2}X_2,\sig Y) \ar[ll]_-{(\xi_1,\sig Y)} & \cdots. \ar[l]
}
\end{equation}
This is dual to the direct system (\ref{equ:newchain}),
$$
\xymatrix{
(\Sigma^nX_0,\Sigma^2Y)^\ast \ar[rr]^-{(\xi_0,\Sigma^2Y)^\ast} &&   (\Sigma^{n-1}X_1,\Sigma^2Y)^\ast \ar[rr]^-{(\xi_1,\Sigma^2Y)^\ast} &&  (\Sigma^{n-2}X_2,\Sigma^2Y)^\ast \ar[r] &  \cdots,
}
$$
which is naturally isomorphic to the direct system (\ref{equ:origchain}), and so has colimit isomorphic to $k$ when $Y \in \BW(\Sigma^nX_0),$ and $0$ when $Y \not\in \BW(\Sigma^nX_0),$ by Lemma \ref{lem:colim}. Therefore, the limit of the inverse system (\ref{equ:inversechain}) is isomorphic to $k$ when $Y \in \BW(\Sigma^nX_0),$ and $0$ when $Y \not\in \BW(\Sigma^nX_0),$ because $(\colim_i(U_i))^\ast \cong \lim_i(U_i^\ast)$. By (\ref{equ:limiso}), we have that $\textrm{lim}(\Sigma^{n-i}X_i,\sig Y)$ is isomorphic to $(\hocolim_i(\Sigma^{n-i}X_i),\sig Y),$ so
$$(\hocolim_i(\Sigma^{n-i}X_i),\sig Y) \cong \Big \{ \begin{matrix} k \hspace{1.5 mm} \textrm{if} \hspace{1.5 mm} Y \in \BW(\Sigma^{n}X_0), \\ 0 \hspace{1.5 mm} \textrm{if} \hspace{1.5 mm} Y \not\in \BW(\Sigma^{n}X_0). \end{matrix}$$
A simple shift of vertices yields the result, namely that
$$(\hocolim_i(\Sigma^{n-i}X_i),Y) \cong \Big \{ \begin{matrix} k \hspace{1.5 mm} \textrm{if} \hspace{1.5 mm} Y \in \BW(\Sigma^{n+2}X_0), \\ 0 \hspace{1.5 mm} \textrm{if} \hspace{1.5 mm} Y \not\in \BW(\Sigma^{n+2}X_0). \end{matrix} $$
\end{proof}

\noindent After shifting vertices, Proposition \ref{pro:homsin} and Corollary \ref{cor:homsout} can be rewritten and summarised in the following way.
\begin{Theorem}\label{rmk:symmetry}
We have the following isomorphisms:
\begin{equation}\label{equ:homsinsym} (Y,\Sigma \sE_n) \cong \Big \{ \begin{matrix} k \hspace{1.5 mm} \textrm{if} \hspace{1.5 mm} Y \in \BW(\Sigma^{n+1}X_0), \\ 0 \hspace{1.5 mm} \textrm{if} \hspace{1.5 mm} Y \not\in \BW(\Sigma^{n+1}X_0) \ \end{matrix} 
\end{equation}
and
\begin{equation}\label{equ:homsoutsym} (\sE_n,\Sigma Y) \cong \Big \{ \begin{matrix} k \hspace{1.5 mm} \textrm{if} \hspace{1.5 mm} Y \in \BW(\Sigma^{n+1}X_0), \\ 0 \hspace{1.5 mm} \textrm{if} \hspace{1.5 mm} Y \not\in \BW(\Sigma^{n+1}X_0). \ \end{matrix} 
\end{equation} \demo
\end{Theorem}

\begin{Theorem}\label{pro:homsbetweenhocolims}
Consider the two direct systems
$$ \Sigma^mX_0 \rightarrow \Sigma^{m-1}X_1 \rightarrow \cdots \rightarrow \Sigma^{m-j}X_j \rightarrow \cdots$$
and
$$ \Sigma^nX_0 \rightarrow \Sigma^{n-1}X_1 \rightarrow \cdots \rightarrow \Sigma^{n-i}X_i \rightarrow \cdots$$
associated with slices in the Auslander-Reiten quiver of $\cD$, with respective hocolimits $\sE_m$ and $\sE_n$ in $\Dbar.$ Then
$$\hmD(\sE_m,\sE_n) \cong \Big \{ \begin{matrix} k \hspace{1.5 mm} \textrm{if} \hspace{2 mm} n \leq m, \\ 0 \hspace{1.5 mm} \textrm{if} \hspace{2 mm} n > m. \end{matrix} $$
\end{Theorem}
 
\begin{proof}
By \citep[Lemma 1.13.1]{JSM}, there is a short exact sequence
$$0 \rightarrow \operatorname{lim}^1_j(\Sigma^{m-j}X_j,\sE_{n-1}) \rightarrow (\hocolim_j(\Sigma^{m-j}X_j),\sE_n) \rightarrow \textrm{lim}_j(\Sigma^{m-j}X_j,\sE_n) \rightarrow 0,$$
whose middle object is $\operatorname{Hom}_{\Dbar}(\sE_m,\sE_n)$. Now, for all $j \in \BN_0,$ $(\Sigma^{m-j}X_j,\sE_{n-1})$ is either isomorphic to $0$ or $k,$ by Proposition \ref{pro:homsin}. Therefore
$${\lim}{}^1_j\big\{(\Sigma^{m-j}X_j,\sE_{n-1} \big\} = 0,$$
by Exercise 3.5.2 of \citep{W}. Therefore,
\begin{equation}\label{equ:hocolimiso}
(\hocolim_j(\Sigma^{m-j}X_j),\sE_n) \cong \textrm{lim}_j(\Sigma^{m-j}X_j,\sE_n).
\end{equation}
Now, apply the functor $(-,\sE_n)$ to the direct system
\begin{equation}\label{equ:chainm}
\xymatrix{
\Sigma^mX_0 \ar[r]^-{\eta_0} & \Sigma^{m-1}X_1 \ar[r]^-{\eta_1} & \Sigma^{m-2}X_2 \ar[r]& \cdots
}
\end{equation}
to obtain the inverse system
\begin{equation}\label{equ:blankcommahocolim}
\xymatrix{
(\Sigma^mX_0,\sE_n) && (\Sigma^{m-1}X_1,\sE_n) \ar[ll]_-{(\eta_0,\sE_n)} && (\Sigma^{m-2}X_2,\sE_n) \ar[ll]_-{(\eta_1,\sE_n)} & \cdots. \ar[l]
}
\end{equation}
Now, two things can happen depending on where the slice associated with the direct system (\ref{equ:chainm}) starts. Consider the regions of the Auslander-Reiten quiver of $\cD$ illustrated below.
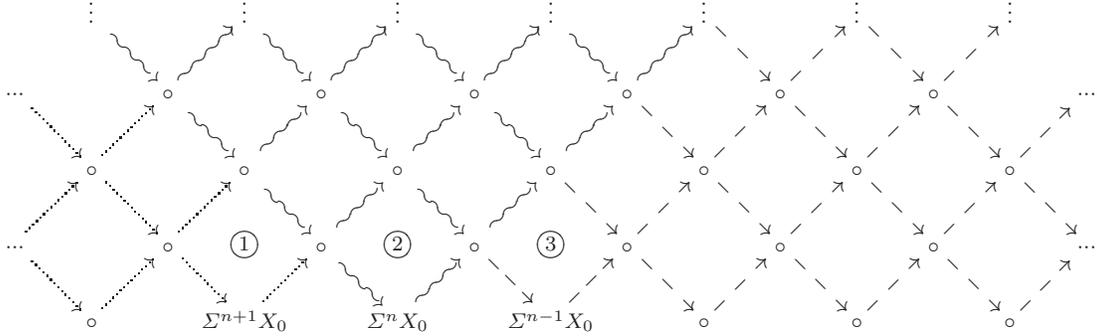
\[ \def\objectstyle{\scriptstyle}
  \xymatrix @-2.9pc @! {
    & \vdots \ar@{~>}[dr] & & \vdots \ar@{~>}[dr] & & \vdots \ar@{~>}[dr] & & \vdots \ar@{~>}[dr] & & \vdots \ar@{-->}[dr] & & \vdots \ar@{-->}[dr] & & \vdots & \\
    \cdots \ar@{.>}[dr]& & \circ \ar@{~>}[ur] \ar@{~>}[dr] & & \circ \ar@{~>}[ur] \ar@{~>}[dr] &  & \circ \ar@{~>}[ur] \ar@{~>}[dr] & & \circ \ar@{~>}[ur] \ar@{-->}[dr] & & \circ \ar@{-->}[ur] \ar@{-->}[dr] & & \circ \ar@{-->}[ur] \ar@{-->}[dr] & & \cdots \\
    & \circ \ar@{.>}[ur] \ar@{.>}[dr] &  & \circ \ar@{~>}[ur] \ar@{~>}[dr] &  & \circ \ar@{~>}[ur] \ar@{~>}[dr] & & \circ \ar@{~>}[ur] \ar@{-->}[dr] &  & \circ \ar@{-->}[ur] \ar@{-->}[dr] &  & \circ \ar@{-->}[ur] \ar@{-->}[dr] &  & \circ \ar@{-->}[ur] \ar@{-->}[dr] & \\
    \cdots \ar@{.>}[ur]\ar@{.>}[dr]& & \circ \ar@{.>}[ur] \ar@{.>}[dr] & *+[o][F]{1} & \circ \ar@{~>}[ur] \ar@{~>}[dr] & *+[o][F]{2} & \circ \ar@{~>}[ur] \ar@{-->}[dr] & *+[o][F]{3} & \circ \ar@{-->}[ur] \ar@{-->}[dr] & & \circ \ar@{-->}[ur] \ar@{-->}[dr] && \circ \ar@{-->}[ur] \ar@{-->}[dr] & & \cdots\\
    & \circ \ar@{.>}[ur] &  & \Sigma^{n+1}X_0 \ar@{.>}[ur] & & \Sigma^{n}X_0 \ar@{~>}[ur] & & \Sigma^{n-1}X_0 \ar@{-->}[ur] & & \circ \ar@{-->}[ur] & & \circ \ar@{-->}[ur] & & \circ \ar@{-->}[ur] & \\
               }
\]
Region $\def\objectstyle{\scriptstyle}\xymatrix{*+[o][F]{2}}$ is highlighted with wavy arrows and contains all of the vertices on its boundary, so, in particular, equals $\BW(\Sigma^{n}X_0).$ Neither region $\def\objectstyle{\scriptstyle}\xymatrix{*+[o][F]{1}},$ illustrated with dotted arrows, nor region $\def\objectstyle{\scriptstyle}\xymatrix{*+[o][F]{3}},$ illustrated with dashed arrows, contain the vertices on their respective shared boundaries with region $\def\objectstyle{\scriptstyle}\xymatrix{*+[o][F]{2}}.$ 

If the slice associated with the direct system (\ref{equ:chainm}) starts in either region $\def\objectstyle{\scriptstyle}\xymatrix{*+[o][F]{1}}$ or region $\def\objectstyle{\scriptstyle}\xymatrix{*+[o][F]{2}}$, so $m = n+x$ with $x \geq 0,$ then it will eventually meet the left boundary of region $\def\objectstyle{\scriptstyle}\xymatrix{*+[o][F]{2}}$ and the inverse system (\ref{equ:blankcommahocolim}) is of the form
\begin{equation} \label{equ:0arrowk}
0 \leftarrow 0 \leftarrow \cdots 0 \leftarrow k \leftarrow k \leftarrow \cdots
\end{equation}
with $x$ zeroes, by Proposition \ref{pro:homsin}. Alternatively, suppose the slice associated with the direct system (\ref{equ:chainm}) starts in region $\def\objectstyle{\scriptstyle}\xymatrix{*+[o][F]{3}},$ so $n>m.$ Then $\Sigma^{m-j}X_j \not\in \BW(\Sigma^nX_0)$ for each $j \in \BN_0,$ so the inverse system (\ref{equ:blankcommahocolim}) is of the form
$$0 \leftarrow 0 \leftarrow 0 \leftarrow \cdots$$
by Proposition \ref{pro:homsin}. Therefore, when $n>m,$ the limit of the inverse system (\ref{equ:blankcommahocolim}) %,
%$$
%\xymatrix{
%(\Sigma^mX_0,\sE_n) && (\Sigma^{m-1}X_1,\sE_n) \ar[ll]_-{(\eta_0,\sE_n)} && (\Sigma^{m-2}X_2,\sE_n) \ar[ll]_-{(\eta_1,\sE_n)} & \cdots, \ar[l]
%}
%$$
is trivially zero. So, suppose $n\leq m.$ To establish the theorem, we must show that the inverse limit of (\ref{equ:blankcommahocolim}) is $k$. Consider one of the nonzero maps in (\ref{equ:blankcommahocolim}). Denote this map
\begin{equation}\label{equ:nonzerohocolimmap}
\xymatrix{
(\Sigma^{m-j}X_j,\sE_n) && (\Sigma^{m-j-1}X_{j+1},\sE_n) \ar[ll]_-{(\eta_j,\sE_n)}
}
\end{equation}
for $j\geq x.$ It is enough to show that each such map is nonzero, because this implies it is an isomorphism, as the hom-spaces are one-dimensional vector spaces. In order to do this, consider the direct system (\ref{equ:slice}),
$$
\xymatrix{
\Sigma^nX_0 \ar[r]^-{\xi_0} & \Sigma^{n-1}X_1 \ar[r]^-{\xi_1} & \Sigma^{n-2}X_2 \ar[r]& \cdots,
}
$$
and apply the functors $(\Sigma^{m-j-1}X_{j+1},-)$ and $(\Sigma^{m-j}X_j,-)$ separately to obtain the respective direct systems
\begin{equation}\label{equ:m-j-1}
\xymatrix{
(\Sigma^{m-j-1}X_{j+1},\Sigma^nX_0) \ar[rrr]^-{(\Sigma^{m-j-1}X_{j+1},\xi_0)} &&& (\Sigma^{m-j-1}X_{j+1},\Sigma^{n-1}X_1) \ar[r]& \cdots
}
\end{equation}
and
\begin{equation}\label{equ:m-j}
\xymatrix{
(\Sigma^{m-j}X_j,\Sigma^nX_0) \ar[rrr]^-{(\Sigma^{m-j}X_j,\xi_0)} &&& (\Sigma^{m-j}X_j,\Sigma^{n-1}X_1) \ar[r]& \cdots.
}
\end{equation}
Both of these direct systems look like the direct system (\ref{equ:origchain}),
$$
\xymatrix{
(Y,\Sigma^nX_0) \ar[r]^-{(Y,\xi_0)} & (Y,\Sigma^{n-1}X_1) \ar[r]^-{(Y,\xi_1)} & (Y,\Sigma^{n-2}X_2) \ar[r]& \cdots,
}
$$ 
with $Y$ replaced with $\Sigma^{m-j-1}X_{j+1}$ and $\Sigma^{m-j}X_j$ respectively. Since both $\Sigma^{m-j-1}X_{j+1}$ and $\Sigma^{m-j}X_j$ are in $\BW(\Sigma^nX_0),$ the direct systems (\ref{equ:m-j-1}) and (\ref{equ:m-j}) both have colimit isomorphic to $k,$ by Proposition \ref{pro:homsin}. We may write their respective colimits as $(\Sigma^{m-j-1}X_{j+1},\sE_n)$ and $(\Sigma^{m-j}X_j,\sE_n).$ The direct systems (\ref{equ:m-j-1}) and (\ref{equ:m-j}) can be combined to yield the commutative ladder
\begin{equation} \begin{aligned} \label{equ:ladder}
\xymatrix{
(\Sigma^{m-j-1}X_{j+1},\Sigma^nX_0) \ar[rrr]^-{(\Sigma^{m-j-1}X_{j+1},\xi_0)} \ar[dd]^-{(\eta_j,\Sigma^nX_0)} &&& (\Sigma^{m-j-1}X_{j+1},\Sigma^{n-1}X_1)  \ar[dd]^-{(\eta_j,\Sigma^{n-1}X_1)} \ar[r]& \cdots \\ \\
(\Sigma^{m-j}X_j,\Sigma^nX_0) \ar[rrr]^-{(\Sigma^{m-j}X_j,\xi_0)} &&& (\Sigma^{m-j}X_j,\Sigma^{n-1}X_1) \ar[r]& \cdots
}
\end{aligned} \end{equation}
which has colimit equal to (\ref{equ:nonzerohocolimmap}). Now, note that the direct systems (\ref{equ:m-j-1}) and (\ref{equ:m-j}) will look like
$$ 0 \rightarrow 0 \rightarrow \dots \rightarrow 0 \rightarrow k \rightarrow k \rightarrow \cdots$$
by Equation (\ref{equ:0arrowk}). Therefore, the ladder (\ref{equ:ladder}) looks like
$$
\xymatrix{
0 \ar[r] \ar[d] & \cdots \ar[r] & 0 \ar[r] \ar[d] & 0 \ar[r] \ar[d] & k \ar[r] \ar[d] & k \ar[r] \ar[d] & \cdots \\
0 \ar[r] & \cdots \ar[r] & 0 \ar[r] & k \ar[r] & k \ar[r] & k \ar[r] & \cdots.
}
$$
For $l \geq j-1,$ the ``step"  
$$
\xymatrix{
(\Sigma^{m-j}X_{j},\Sigma^{n-l}X_l) && (\Sigma^{m-j-1}X_{j+1},\Sigma^{n-l}X_l) \ar[ll]_-{(\eta_j,\Sigma^{n-l}X_l)}
}
$$
in the ladder (\ref{equ:ladder}) is nonzero by \citep[Lemma 2.5]{J}. For example, for $l=j-1,$ we can illustrate this on the Auslander-Reiten quiver as follows.
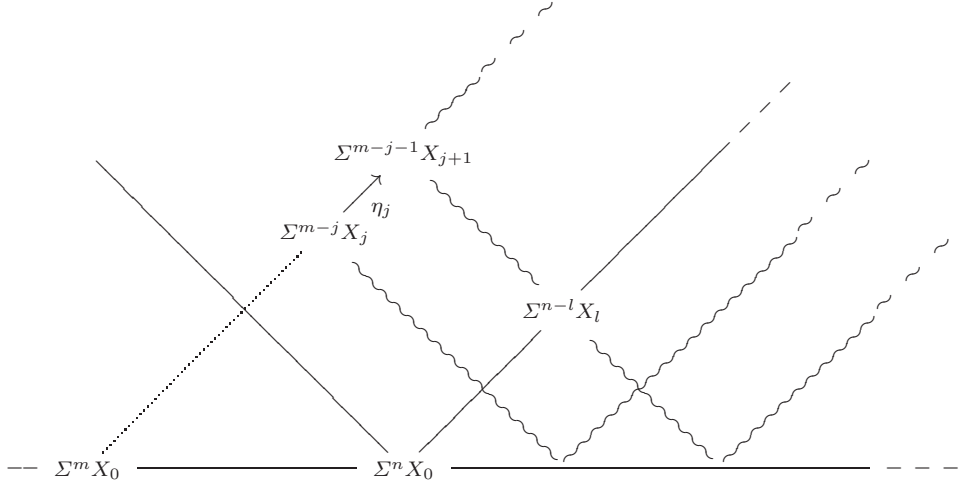
\[ \def\objectstyle{\scriptstyle}
  \xymatrix @-4.5pc @! {
    & & & & & & & & *{} \\
    & & & & & & *{} \ar@{~~}[ur] & & & & \\
    &  \ar@{-}[ddddrrrr] & & & & \Sigma^{m-j-1}X_{j+1} \ar@{~}[ur] \ar@{~}[ddrr] & &  &  & *{}\ar@{--}[ur] & &  \\
    & & & & \Sigma^{m-j}X_{j} \ar@{..}[dddlll] \ar@{~}[dddrrr] \ar@{->}[ur]_-{\eta_j} & & & & & *{} & *{}\ar@{~~}[ur] & &  \\
    *{} & & & & &&& \Sigma^{n-l}X_l \ar@{-}[uurr]  \ar@{~}[ddrr] & & & & *{}\ar@{~~}[ur] &\\ 
    & & & & & & & *{} & & & &  \\
    *{} \ar@{--}[r] & {\hspace{1mm} \Sigma^mX_0 \hspace{1mm} } \ar@{-}[rrrr] & &  & & {\hspace{1mm} \Sigma^nX_0 \hspace{1mm} } \ar@{-}[uurr] \ar@{-}[rrrrrr] & &*{}\ar@{~}[uuurrr]&  & *{} \ar@{~}[uurr] &&\ar@{--}[r]&*{}\\
           }
\]
Hence the colimit of the ladder (\ref{equ:ladder}) is a nonzero map, which means the map (\ref{equ:nonzerohocolimmap}),
$$
\xymatrix{
(\Sigma^{m-j}X_j,\sE_n) && (\Sigma^{m-j-1}X_{j+1},\sE_n) \ar[ll]_-{(\eta_j,\sE_n)},
}
$$
is nonzero as desired.
\end{proof}

\begin{Corollary} \label{cor:indecomphoc}
The hocolimits $\sE_n$ of $\Dbar$ are indecomposable objects.
\end{Corollary}

\begin{proof}
By Theorem \ref{pro:homsbetweenhocolims}, we have $\hmD(\sE_n,\sE_n) \cong k.$ Hence, $\hmD(\sE_n,\sE_n)$ is a one-dimensional vector space over $k,$ and is therefore a local ring. Thus, $\sE_n$ is an indecomposable object, as required.
%the above needs a reference
\end{proof}

\section{The geometric model of $\Dbar$}
\label{sec:geometric}

Theorem \ref{rmk:symmetry} states that 
\begin{equation} \label{equ:symmetry2} \hmD(Y,\Sigma \sE_n) \cong \hmD(\sE_n,\Sigma Y) \cong \Big \{ \begin{matrix} k \hspace{1.5 mm} \textrm{if} \hspace{1.5 mm} Y \in \BW(\Sigma^{n+1}X_0), \\ 0 \hspace{1.5 mm} \textrm{if} \hspace{1.5 mm} Y \not\in \BW(\Sigma^{n+1}X_0). \end{matrix} 
\end{equation} The indecomposables corresponding to $\BW(\Sigma^{n+1}X_0)$ are highlighted in the Auslander-Reiten quiver of $\cD$ below. 
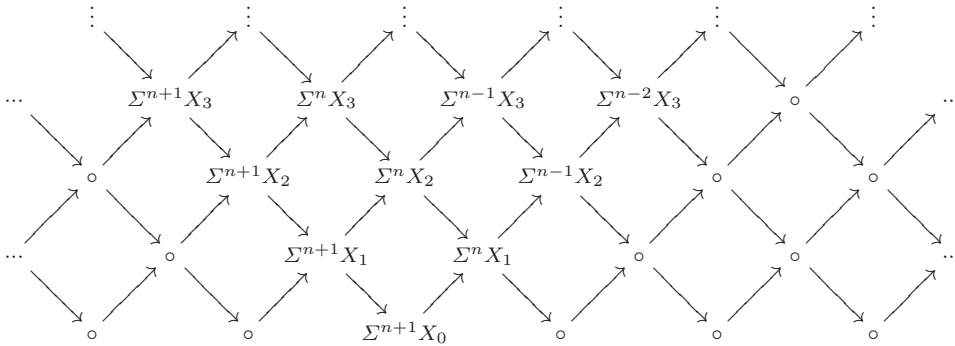
\[ \def\objectstyle{\scriptstyle}
  \xymatrix @-2.85pc @! {
    & \vdots \ar[dr] & & \vdots \ar[dr] & & \vdots \ar[dr] & & \vdots \ar[dr] & & \vdots \ar[dr] & & \vdots & \\
    \cdots \ar[dr]& & \Sigma^{n+1} X_3 \ar[ur] \ar[dr] & & \Sigma^{n} X_3 \ar[ur] \ar[dr] & & \Sigma^{n-1} X_3 \ar[ur] \ar[dr] & & \Sigma^{n-2} X_3 \ar[ur] \ar[dr] & & \circ \ar[ur] \ar[dr] & & \cdots \\
    & \circ \ar[ur] \ar[dr] & & \Sigma^{n+1} X_2 \ar[ur] \ar[dr] & & \Sigma^{n} X_2 \ar[ur] \ar[dr] & & \Sigma^{n-1} X_2 \ar[ur] \ar[dr] & & \circ \ar[ur] \ar[dr] & & \circ \ar[ur] \ar[dr] & \\
    \cdots \ar[ur]\ar[dr]& & \circ \ar[ur] \ar[dr] & & \Sigma^{n+1} X_1 \ar[ur] \ar[dr] & & \Sigma^{n} X_1 \ar[ur] \ar[dr] & & \circ \ar[ur] \ar[dr] & & \circ \ar[ur] \ar[dr] & & \cdots\\
    & \circ \ar[ur] & & \circ \ar[ur] & & \Sigma^{n+1} X_0 \ar[ur] & & \circ \ar[ur] & & \circ \ar[ur] & & \circ \ar[ur] & \\
               }
\]
The formula
$$
\Sigma^{n-l}X_{l-k-2} = (-n+k,-n+l)
$$
defines the ``standard'' coordinate system on the Auslander-Reiten quiver of $\cD$ given in Remark 1.4 of \citep{J}. This is illustrated below.
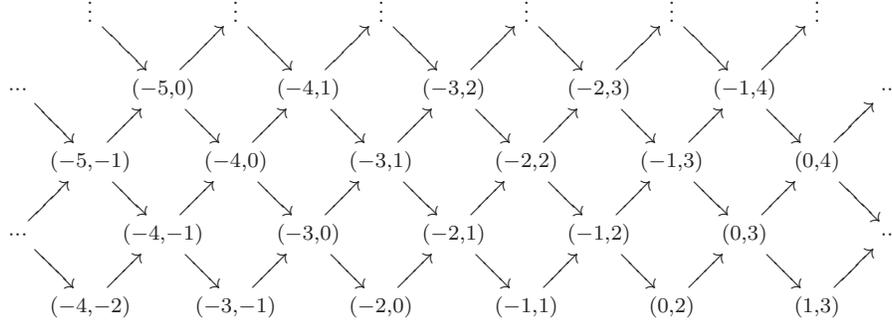
\[ \def\objectstyle{\scriptstyle}
  \xymatrix @-2.85pc @! {
    & \vdots \ar[dr] & & \vdots \ar[dr] & & \vdots \ar[dr] & & \vdots \ar[dr] & & \vdots \ar[dr] & & \vdots & \\
    \cdots \ar[dr]& & (-5,0) \ar[ur] \ar[dr] & & (-4,1) \ar[ur] \ar[dr] & & (-3,2) \ar[ur] \ar[dr] & & (-2,3) \ar[ur] \ar[dr] & & (-1,4) \ar[ur] \ar[dr] & & \cdots \\
    & (-5,-1) \ar[ur] \ar[dr] & & (-4,0) \ar[ur] \ar[dr] & & (-3,1) \ar[ur] \ar[dr] & & (-2,2) \ar[ur] \ar[dr] & & (-1,3) \ar[ur] \ar[dr] & & (0,4) \ar[ur] \ar[dr] & \\
    \cdots \ar[ur]\ar[dr]& & (-4,-1) \ar[ur] \ar[dr] & & (-3,0) \ar[ur] \ar[dr] & & (-2,1) \ar[ur] \ar[dr] & & (-1,2) \ar[ur] \ar[dr] & & (0,3) \ar[ur] \ar[dr] & & \cdots\\
    & (-4,-2) \ar[ur] & & (-3,-1) \ar[ur] & & (-2,0) \ar[ur] & & (-1,1) \ar[ur] & & (0,2) \ar[ur] & & (1,3) \ar[ur] & \\
               }
\]
Pairs $(i,j)$ with $i \leq j-2$ can be viewed as both coordinate pairs and arcs between non-neighbouring integers on the number line. For example, the coordinate pair associated with $\Sigma^{n+1} X_0$ is $(-n-3,-n-1),$ and by Remark 3.4 of \citep{J}, this would be associated with an arc drawn between the integers $-n-3$ and $-n-1$ on the number line. We illustrate below the arcs which correspond to to the indecomposable objects of $\BW(\Sigma^{n+1}X_0).$ \\

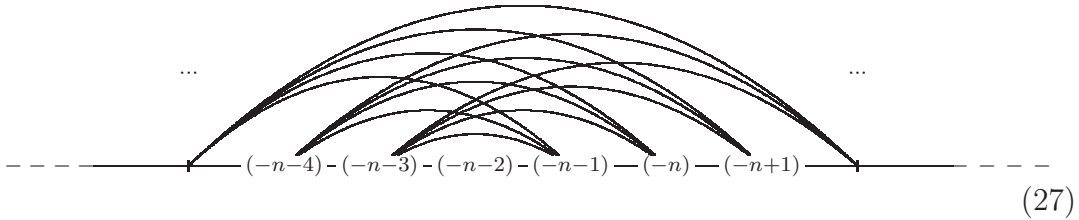
\begin{equation} \begin{aligned} \label{equ:crossings1} \def\objectstyle{\scriptstyle}
\vcenter{
  \xymatrix @-2.8pc @! { && \cdots &&&&&&& \cdots \\
       \rule{0ex}{7.5ex} \ar@{--}[r]
     & *{}\ar@{-}[r]
     & *{\rule{0.1ex}{0.8ex}} \ar@{-}[r] \ar@/^2.8pc/@{-}[rrrr] \ar@/^3.55pc/@{-}[rrrrr] \ar@/^4.3pc/@{-}[rrrrrr] \ar@/^5.05pc/@{-}[rrrrrrr]
     & *{\: (-n-4) \:} \ar@{-}[r] \ar@/^1.75pc/@{-}[rrr] \ar@/^2.65pc/@{-}[rrrr] \ar@/^3.4pc/@{-}[rrrrr] \ar@/^4.15pc/@{-}[rrrrrr]
     & *{\: (-n-3) \:} \ar@{-}[r] \ar@/^1.0pc/@{-}[rr] \ar@/^1.75pc/@{-}[rrr] \ar@/^2.5pc/@{-}[rrrr] \ar@/^3.25pc/@{-}[rrrrr]
     & *{\: (-n-2) \:} \ar@{-}[r]
     & *{\: (-n-1) \:} \ar@{-}[r] 
     & *{\: (-n) \:} \ar@{-}[r] 
     & *{\: (-n+1) \:} \ar@{-}[r]
     & *{\rule{0.1ex}{0.8ex}} \ar@{-}[r]
     & *{}\ar@{--}[r]
     & *{} \\
                    }
}
\end{aligned}
\end{equation}
They are, in fact, all the `overarcs' of the integer $-n-2.$ Lemma 3.6 of \citep{J} states that if $\fa, \fb$ are arcs with $a$ and $b$ the corresponding indecomposables, then $(a,\Sigma b) \cong (b,\Sigma a) \cong k$ if and only if $\fa$ and $\fb$ cross; if $\fa$ and $\fb$ don't cross, then  $(a,\Sigma b) \cong (b,\Sigma a) \cong 0.$ \\

\noindent To extend this arc model to $\Dbar,$ we formulate the following definition.

\begin{Definition}
An arc is either (1) a pair of integers $(i,j)$ with $i \leq j-2,$ which will be referred to as a {\it{finite arc}}, or (2) a pair $(i,\infty)$ with $i$ an integer, which will be referred to as an {\it{infinite arc}}. Two finite arcs $(i,j), (r,s)$ cross if either $i < r < j < s$ or $r < i < s < j.$ A finite arc $(i,j)$ crosses an infinite arc $(n,\infty)$ if $i < n < j.$ There is no meaning associated with two infinite arcs crossing. \demo
\end{Definition}

\begin{Remark}
Let $a,b,n \in \BZ.$ Then the arcs $(a,b)$ and $(n,\infty)$ cross if $a < n < b.$ That is, we explictly rule out the possibility that we may draw an arc which passes ``over'' the arc to infinity and does not cross it. This is illustrated below. \\

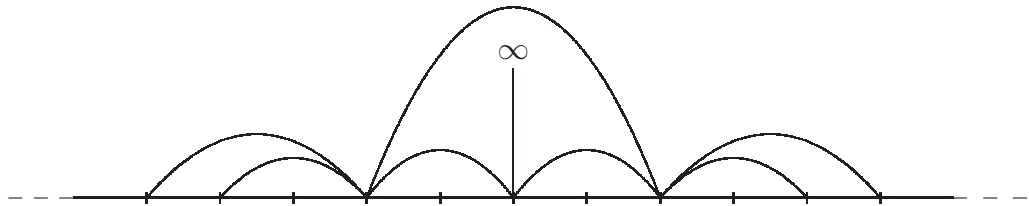
\[
\vcenter{
  \xymatrix @-3.5pc @! {
&&&&&&&\infty \ar@{-}[dd] && \\ \\ 
       \rule{0ex}{7.5ex} \ar@{--}[r]
     & *{}\ar@{-}[r]
     & *{\rule{0.1ex}{0.8ex}} \ar@{-}[r] \ar@/^2.0pc/@{-}[rrr]
     & *{\rule{0.1ex}{0.8ex}} \ar@{-}[r] \ar@/^1.25pc/@{-}[rr]
     & *{\rule{0.1ex}{0.8ex}} \ar@{-}[r]
     & *{\rule{0.1ex}{0.8ex}} \ar@{-}[r] \ar@/^6.0pc/@{-}[rrrr] \ar@/^1.5pc/@{-}[rr]
     & *{\rule{0.1ex}{0.8ex}} \ar@{-}[r] 
     & *{\rule{0.1ex}{0.8ex}} \ar@{-}[r] 
     & *{\rule{0.1ex}{0.8ex}} \ar@{-}[r]
     & *{\rule{0.1ex}{0.8ex}} \ar@{-}[r] \ar@/^-1.5pc/@{-}[ll]
     & *{\rule{0.1ex}{0.8ex}} \ar@{-}[r]
     & *{\rule{0.1ex}{0.8ex}} \ar@{-}[r] \ar@/^-1.25pc/@{-}[ll]
     & *{\rule{0.1ex}{0.8ex}} \ar@{-}[r] \ar@/^-2.0pc/@{-}[lll]
     & *{}\ar@{--}[r]
     & *{}
                    }
}
\]
Such an arc diagram is not permitted! \demo
\end{Remark}

\noindent The sketches with arcs are simply visualisations of sets of arcs. The infinite arcs `represent' the objects $\sE_n$ by the following extension of \citep[Lemma 3.6]{J}.

\begin{Proposition} \label{pro:infinitearc} 
Let $\sE_n$ be a hocolimit in $\Dbar.$ Let $\fY$ be an arc representing a finite indecomposable object of $\Dbar$ and consider the infinite arc $\fE_n = (-n-2,\infty)$. Then $\hmD(Y,\Sigma \sE_n) \cong \hmD(\sE_n,\Sigma Y) \cong k$ if and only if $\fE_n$ and $\fY$ cross.
\end{Proposition}

\begin{proof}
The arc $\fE_n$ crosses precisely the arcs drawn in Figure (\ref{equ:crossings1}); this is illustrated below.
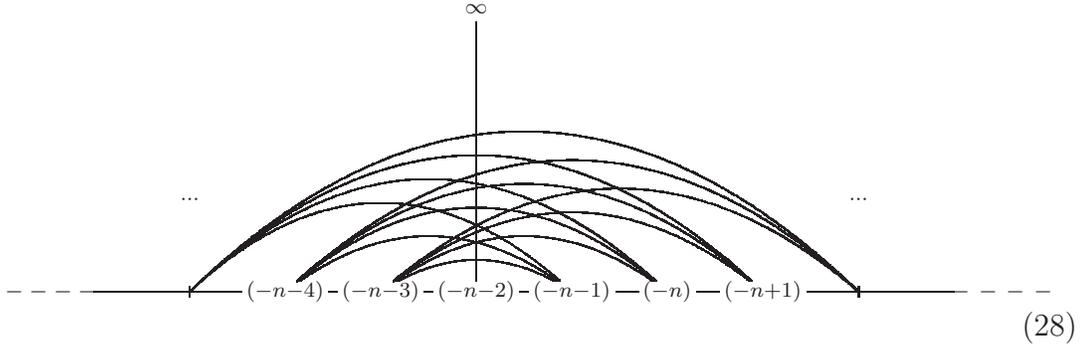
\begin{equation} \begin{aligned} \label{equ:crossings2} \def\objectstyle{\scriptstyle}
\vcenter{
  \xymatrix @-2.8pc @! { &&&&&\infty \ar@{-}[ddd] \\ \\ && \cdots &&&&&&& \cdots \\
       \rule{0ex}{7.5ex} \ar@{--}[r]
     & *{}\ar@{-}[r]
     & *{\rule{0.1ex}{0.8ex}} \ar@{-}[r] \ar@/^2.8pc/@{-}[rrrr] \ar@/^3.55pc/@{-}[rrrrr] \ar@/^4.3pc/@{-}[rrrrrr] \ar@/^5.05pc/@{-}[rrrrrrr]
     & *{\: (-n-4) \:} \ar@{-}[r] \ar@/^1.75pc/@{-}[rrr] \ar@/^2.65pc/@{-}[rrrr] \ar@/^3.4pc/@{-}[rrrrr] \ar@/^4.15pc/@{-}[rrrrrr]
     & *{\: (-n-3) \:} \ar@{-}[r] \ar@/^1.0pc/@{-}[rr] \ar@/^1.75pc/@{-}[rrr] \ar@/^2.5pc/@{-}[rrrr] \ar@/^3.25pc/@{-}[rrrrr]
     & *{\: (-n-2) \:} \ar@{-}[r]
     & *{\: (-n-1) \:} \ar@{-}[r] 
     & *{\: (-n) \:} \ar@{-}[r] 
     & *{\: (-n+1) \:} \ar@{-}[r]
     & *{\rule{0.1ex}{0.8ex}} \ar@{-}[r]
     & *{}\ar@{--}[r]
     & *{} \\
                    }
}
\end{aligned}
\end{equation}
Hence $\fE_n$ crosses precisely the arcs $\fY$ which correspond to the indecomposable objects $Y$ of $\BW(\Sigma^{n+1}X_0),$ so Theorem \ref{rmk:symmetry} says that $\hmD(Y,\Sigma \sE_n) \cong \hmD(\sE_n,\Sigma Y) \cong k$ if and only if $\fE_n$ and $\fY$ cross.
\end{proof}

\begin{Remark} \label{rmk:nonsymhoc}
%Lemma 3.6 of \citep{J} and Theorem \ref{pro:homsbetweenhocolims} are not compatible; indeed
By Theorem \ref{pro:homsbetweenhocolims}, we have that
$$\hmD(\sE_m,\Sigma\sE_n) \cong \Big \{ \begin{matrix} k \hspace{1.5 mm} \textrm{if} \hspace{2 mm} m-1 \geq n, \\ 0 \hspace{1.5 mm} \textrm{if} \hspace{2 mm} m-1 < n, \end{matrix} $$
and 
$$\hmD(\sE_n,\Sigma\sE_m) \cong \Big \{ \begin{matrix} k \hspace{1.5 mm} \textrm{if} \hspace{2 mm} n \geq m+1, \\ 0 \hspace{1.5 mm} \textrm{if} \hspace{2 mm} n < m+1. \end{matrix} $$
Together, these imply 
$$\hmD(\sE_m,\Sigma\sE_n) \cong \hmD(\sE_n,\Sigma\sE_m) \cong 0$$ if $m-1 < n < m+1,$ which is only possible if $n=m,$ and 
$$\hmD(\sE_m,\Sigma\sE_n) \cong \hmD(\sE_n,\Sigma\sE_m) \cong k$$ if $m-1 \geq n \geq m+1,$ which is impossible. Hence, it is impossible to devise a symmetric notion of crossing of the infinite arcs $\fE_m$ and $\fE_n$ which corresponds to non-vanishing of the Hom-spaces in the equations. Indeed, the isomorphisms in (\ref{equ:symmetry2}) display symmetry, whilst the isomorphism of Theorem \ref{pro:homsbetweenhocolims} is inherently non-symmetrical. \demo
\end{Remark}

\section{Weakly cluster tilting subcategories}
\label{sec:wcts}

A series of arc lemmas is necessary to prove the subsequent results about (weakly) cluster tilting subcategories of $\Dbar.$ In this section we also look at subcategories of $\Dbar$ with special arc configurations and show when they correspond to weakly cluster tilting subcategories.

\noindent By \citep{I}, a weakly cluster tilting subcategory $\cT$ of a triangulated category is defined by $$X \in \cT \hspace{1.5mm} \textrm{if and only if} \hspace{1.5mm} (\cT,\Sigma X) \cong (X, \Sigma \cT) \cong 0.$$ A weakly cluster tilting subcategory is closed under direct sums and summands, so if the ambient triangulated category is $\Dbar$, it is determined by its indecomposable objects, i.e. by the sets of arcs corresponding to its indecomposables.

\begin{Proposition}\label{pro:wct}
A weakly cluster tilting subcategory $\cT$ of $\Dbar$ may contain at most one of the objects $\sE_n$. 
\end{Proposition}

\begin{proof} If $\cT$ is weakly cluster tilting, then $\sE_n \in \cT$ implies $(\cT,\Sigma \sE_n) \cong (\sE_n, \Sigma \cT) \cong 0.$ If $\sE_m$ is also in $\cT,$ then by Remark \ref{rmk:nonsymhoc} we have that $(\sE_m,\Sigma \sE_n) \cong (\sE_n, \Sigma \sE_m) \cong 0$ implies $m=n.$ 
\end{proof}

\begin{Definition} \label{def:fountains}
Fountains are originally defined in \citep[Definition 3.2]{J}. We recall their definitions. Let $\fT$ be a set of finite arcs. If there are infinitely many arcs of the form $(m,\text{---})$ in $\fT$, then $m$ is called a {\textit{right-fountain}} of $\fT$. Conversely if there are infinitely many arcs of the form $(\text{---},n)$ in $\fT$, then $n$ is called a {\textit{left-fountain}} of $\fT$. If $m=n$ is both a left- and a right-fountain of $\fT$, then it is a fountain. \demo
\end{Definition}

\begin{Definition} \label{def:locallyfiniteconfig}
Let $\fT$ be a set of finite arcs. If for all $n \in \BZ$ there are only finitely many arcs of the form $(n,\text{---})$ in $\fT$, and also only finitely many arcs of the form $(\text{---},n)$ in $\fT$, then $\fT$ is called {\textit{locally finite}}. \demo
\end{Definition}

\begin{Definition} \label{def:strongoverarc}
Let $\fT$ be a maximal, non-crossing set of finite arcs. If $(p,q) \in \fT,$ then a {\textit{strong overarc}} of $(p,q)$ is a finite arc $(x,y)$ in $\fT$ where $x<p<q<y.$ We also define a strong overarc of an integer $n$. If $n \in \BZ$ then a strong overarc (with respect to $\fT$) of $n$ is an arc $(x,y)$ in $\fT$ where $x<n<y.$ \demo
\end{Definition}

\begin{Lemma}\label{lem:noncross}
Let $\fT'$ be a set of finite arcs and let $\fT = \{(m,\infty)\}\cup \fT'$ be a set of arcs which satisfies the following condition: if a finite arc $\fa$ crosses neither $(m,\infty)$ nor an arc in $\fT'$, then $\fa \in \fT$. Then \\
(1) $\fT'$ has a left-fountain $p$ and a right-fountain $q$ (necessarily with $p \leq m \leq q$), and \\
(2)  Either $p \in \{m-1,m\}$ or $(p,m) \in \fT$ (and symmetrically, either $q \in \{m,m+1\}$ or $(m,q) \in \fT$).
\end{Lemma}

\begin{proof} \*{} \\
By symmetry, it is enough to find a left-fountain $p$ and show that either $p \in \{m-1,m\}$ or $(p,m) \in \fT$. There are seven cases to check. \\
Case $(1)$. Suppose there are infinitely many arcs of the form $(\text{---},m)$ in $\fT'$. Then we have that $p=m$ is a left-fountain. \\
Case $(2)$. Suppose that there are only finitely many arcs of the form $(\text{---},m)$ in $\fT'$ and that $(l,m)$ is the longest such arc. Furthermore, suppose that there are infinitely many arcs in $\fT'$ of the form $(\text{---},l)$. Then we have that $p=l$ is a left-fountain, and the finite arc $(p,m)$ is in $\fT,$ by assumption. \\
Case $(3)$. Suppose that there are only finitely many arcs of the form $(\text{---},m)$ in $\fT'$ and that $(l,m)$ is the longest such arc. Furthermore, suppose that there are only finitely many arcs in $\fT'$ of the form $(\text{---},l)$, and that $(v,l)$ is the longest. But now the arc $(v,m)$ crosses no arc in $\fT'$ nor the arc $(m,\infty)$, and is hence in $\fT$, contradicting the fact that $(l,m)$ is the longest arc of the form $(\text{---},m)$ in $\fT'$. \\
Case $(4)$. Suppose that there are only finitely many arcs of the form $(\text{---},m)$ in $\fT'$ and that $(l,m)$ is the longest such arc. Furthermore, suppose that there is no arc in $\fT'$ of the form $(\text{---},l)$. Then the arc $(l-1,m)$ crosses no arc in $\fT'$ nor the arc $(m,\infty)$, and is hence in $\fT$, contradicting the fact that $(l,m)$ is the longest arc of the form $(\text{---},m)$ in $\fT'$. \\
Case $(5)$. Suppose that there is no arc in $\fT'$ of the form $(\text{---},m)$, and there is also no arc in $\fT'$ of the form $(\text{---},m-1).$ Then this is a contradiction, for the arc $(m-2,m)$ is now necessarily in $\fT,$ for it crosses no finite arc in $\fT'$ nor the arc $(m,\infty).$ \\
Case $(6)$. Suppose that there is no arc in $\fT'$ of the form $(\text{---},m)$, and that there are infinitely many arcs in $\fT'$ of the form $(\text{---},m-1).$ Then we have that $p=m-1$ is a left-fountain. \\
Case $(7)$. Suppose that there is no arc in $\fT'$ of the form $(\text{---},m)$, and that there are only finitely many arcs in $\fT'$ of the form $(\text{---},m-1).$ Let $(l,m-1)$ be the longest such arc. Then $(l,m)$ crosses no arc in $\fT'$ nor the arc $(m,\infty)$ and is hence in $\fT.$ This is a contradiction, because we assumed that there was no arc in $\fT'$ of the form $(\text{---},m)$.

Note, in this proof, the arc $(m,\infty)$ plays a vital ``blocking'' role. Take, for example, Case (5). The fact that there is no arc in $\fT$ of the form $(m-1,\text{---})$ is what permits us to conclude that $(m-2,m)$ is in $\fT$ (for if there was a finite arc of the form $(m-1,\text{---})$ in $\fT$, then it would necessarily cross $(m,\infty)$, which would be a contradiction).
\end{proof}

\begin{Lemma} \label{lem:pminus1}
Let $\fT$ be a maximal, non-crossing, locally finite set of finite arcs. Let $p$ be an arbitrary integer. Suppose $\nexists x \in \BZ$ such that $(x,p) \in \fT.$ Then $\exists y' \in \BZ$ such that $(p-1,y') \in \fT.$
\end{Lemma}

\begin{proof}
First, suppose $\not\exists x' \in \BZ$ such that $(p,x') \in \fT.$ Then there is neither an arc $(x,p)$ nor an arc $(p,x')$ in $\fT.$ Because no arcs end in $p,$ there is room for $(p-1,p+1)$ in the configuration. By maximality of $\fT,$ we must therefore have $(p-1,p+1) \in \fT.$ Therefore the claim is satisfied with $y' = p+1.$
Now suppose $\exists x' \in \BZ$ such that $(p,x') \in \fT.$ Because $\fT$ is locally finite (and hence, in particular, does not contain a right fountain) there is a longest arc of the form $(p,r) : r \in \BZ,$ say, $(p, x'').$ Now, there is room in the configuration for $(p-1,x'')$ because there are no arcs of the form $(x,p) : x \in \BZ$ to ``block'' such an arc. And because $\fT$ is maximal, we must therefore have that $(p-1,x'') \in \fT.$ Hence the claim is again satisfied with $y' = x''.$     
\end{proof}

\noindent By symmetry, we have the following corollary.

\begin{Corollary} \label{cor:qplus1}
Let $\fT$ be a maximal, non-crossing, locally finite set of finite arcs. Suppose $\not\exists x \in \BZ$ such that $(q,x) \in \fT.$ Then $\exists y' \in \BZ$ such that $(y',q+1) \in \fT.$
\end{Corollary}

\begin{Lemma} \label{lem:pqnotlongest}
Let $\fT$ be a maximal, non-crossing set of finite arcs. Let $(p,q)$ be the longest arc in $\fT$ of the form $(p,r) : r \in \BZ$ and also the longest arc in $\fT$ of the form $(l,q) : l \in \BZ.$ Then $\fT$ has a right fountain.
\end{Lemma}

\begin{proof}
Suppose $q$ is not a right fountain. Then either

(i) there is no arc of the form $(q,s) : s \in \BZ,$ or

(ii) there is an arc of the form $(q,s) : s \in \BZ$ and hence a longest arc of this form, say, $(q,s').$

In case (i), there is an arc in $\fT$ of the form $(l',q+1) : l \in \BZ$ by Corollary \ref{cor:qplus1}. By assumption, $l' \neq p,$ else $(p,q+1)$ is an arc in $\fT$ longer than $(p,q).$ So there are two arcs $(l',q+1)$ and $(p,q)$ in $\fT.$ Now, $(p,q+1) \not\in \fT$ by assumption, whence $(l',q) \in \fT$ by maximality of $\fT$: a contradiction.

In the second case (ii), there is a longest arc $(q,s') \in \fT.$ By assumption, $(p,s')\not\in\fT$ and by maximality of $\fT$ the only way this is possible if there is an arc of the form $(l',q)$ with $l'<p,$ because then the arc $(p,s')$ would be ``blocked''. In this case, $(l',q)$ is longer then $(p,q)$: a contradiction.

Now suppose $q$ is a right fountain. Then $\fT$ has a right fountain: $q$ itself.
\end{proof}

\begin{Lemma} \label{lem:locallyfinitemainclaim}
Let $\fT$ be a maximal, non-crossing, locally finite set of finite arcs. Let $(p,q) \in \fT.$ Then $\exists (x,y) \in \fT$ with $x<p<q<y.$
\end{Lemma}

\begin{proof}
By Lemma \ref{lem:pqnotlongest}, if $(p,q)$ is both the longest arc in $\fT$ of the form  $(p,r) : r \in \BZ$ and also the longest arc in $\fT$ of the form $(l,q) : l \in \BZ,$ then $\fT$ has a right fountain and in particular is not locally finite. So either

(i) there is an arc $(p,q')$ in $\fT$ with $q'>q,$ or

(ii) there is an arc $(p',q)$ in $\fT$ with $p'<p,$

but not both, else $(p',q)$ and $(p,q')$ cross. Notice that, if the claim is proved assuming (i) holds, then by symmetry, the claim would be proved if it were the case that (ii) holds. So suppose it is (i) that holds. Note that $p$ is not a right fountain, because $\fT$ is locally finite. So there is a longest arc of the form $(p,r) : r \in \BZ,$ say, $(p,y).$ By assumption, $y \neq q.$ Because $(p,y)$ is the longest arc of the form $(p,r) : r \in \BZ,$ and $\fT$ is locally finite, $(p,y)$ is not the longest arc of the form $(u,y) : u \in \BZ,$ and because $y$ is not a left fountain, there is a longest arc $(x,y) \in \fT$ of this form. Hence, if (i) holds, we are done: $(x,y) \in \fT$ with $x<p<q<y,$ and as noted before, if (i) proves the claim then so does (ii) by symmetry.
\end{proof}

\begin{Corollary} \label{cor:arbitrarilylongarc}
Let $\fT$ be a maximal, non-crossing, locally finite set of finite arcs. Let $(p,q) \in \fT.$ Then there exists a strong overarc $(p-\varepsilon,q+\delta)$ in $\fT$ with $\varepsilon, \delta$ arbitrarily large integers.
\end{Corollary}

\begin{proof}
By Lemma \ref{lem:locallyfinitemainclaim}, $(p,q)$ has a strong overarc in $\fT.$ But any strong overarc of $(p,q)$ will itself have a strong overarc in $\fT$ - which itself will be an (even longer) overarc of $(p,q).$ This can be continued indefinitely with the overarcs of $(p,q)$ becoming arbitrarily long, hence proving the claim. 
\end{proof}

\begin{Corollary} \label{cor:locallyfinitesecondclaim}
Let $\fT$ be a maximal, non-crossing, locally finite set of finite arcs. Let $h \in \BZ.$ Then $\exists (x,y) \in \fT$ with $x<h<y.$
\end{Corollary}

\begin{proof}
Let $(p,q) \in \fT$ be any arc in $\fT.$ Then by Corollary \ref{cor:arbitrarilylongarc}, there is a (potentially very long) strong overarc $(x,y)$ of $(p,q)$ in $\fT$ with 
$$h \in \{x+1,\cdots,p,p+1,\cdots,q,q+1,\cdots,y-1\},$$
which proves the claim. 
\end{proof}

\begin{Theorem}\label{pro:maximal}
Let $\cT$ be a a subcategory of $\Dbar$ which is closed under direct sums and direct summands, so it corresponds to a set $\fT$ of arcs (finite and infinite). Then $\cT$ is weakly cluster tilting if and only if one of the following happens: \\
(1) $\fT$ is a set of finite arcs which is maximal, non-crossing and locally finite in the sense of \citep[Section 3]{J}, or \\
(2) $\fT = \{(m,\infty)\} \cup \fT'$ where $\fT'$ is a set of finite arcs which is maximal non-crossing and has a fountain in the sense of \citep[Section 3]{J} (and the fountain is necessarily at $m$).
\end{Theorem}

\begin{proof} \*{} \\
$[\Rightarrow]$ Suppose that $\cT$ is a weakly cluster tilting subcategory of $\Dbar$, i.e., $\cT = \{t \hspace{1.5mm} | \hspace{1.5mm} (t,\Sigma \cT) = 0 \}$ and $\cT = \{t \hspace{1.5mm} | \hspace{1.5mm} (\cT, \Sigma t) = 0 \}$. By Proposition \ref{pro:wct}, $\cT$ contains either no $\sE_n$ or it contains precisely one $\sE_n.$

If $\cT$ contains no $\sE_n$, then $\fT$ is a set of finite arcs. The set $\fT$ is clearly non-crossing, because if any two arcs in $\fT$ cross, say $\fa$ and $\fb,$ then the objects associated with these arcs have a non-vanishing Ext, contradicting the fact that $\cT = \{t \hspace{1.5mm} | \hspace{1.5mm} (t,\Sigma \cT) = 0 \}.$ The set $\fT$ is also maximal, because if $\fa$ crosses no arc in $\fT$, then $(a,\Sigma \cT) = 0,$ whence $a \in \cT$ so $\fa \in \fT.$ Note that $\fT$ couldn't have a fountain at $-m-2$, because if it did, then $(-m-2,\infty)$ would cross no arc in $\fT$ whence $(\sE_m,\Sigma \cT) = 0,$ so $\sE_m \in \cT$ and $(-m-2,\infty) \in \fT$ which contradicts the fact that $\fT$ is a set of finite arcs. Altogether this gives that $\fT$ is a maximal, non-crossing, locally finite set of finite arcs.

Suppose that $\cT$ contains a single $\sE_n$, so $\fT$ contains a single infinite arc, i.e. $\fT = \{(m,\infty)\}\cup \fT'$ where $\fT'$ consists of finite arcs. Define $\cT'$ to be the category with objects in the additive hull of the corresponding (finite) indecomposables. We know that if a finite arc $\fa$ crosses neither $(m,\infty)$ nor an arc in $\fT'$, then $\fa \in \fT'$, because $\cT$ is weakly cluster tilting. Hence, by the first part of Lemma \ref{lem:noncross}, $\fT'$ has a left-fountain (call it $p$) and a right-fountain (call it $q$). In particular, since $\cT$ is weakly cluster tilting, no two arcs in $\fT$ can cross, so $p \leq m \leq q$. Now suppose $p<m$. Then by the second part of Lemma \ref{lem:noncross}, $p=m-1$ or $(p,m) \in \fT'.$ Since no two arcs in $\fT$ cross, $p$ has no strong overarc in $\fT'$. Now let $F = \sE_{-p-2} \in \Dbar$ be the object associated with the arc $(p,\infty).$ Since $p$ has no strong overarc in $\fT'$, we get $(p,\infty)$ crossing no arc in $\fT'$, so $(\cT',\Sigma F) = (F, \Sigma \cT') = 0$. However, we also have $(\sE_{-m-2},\Sigma F) = 0$ as $p < m.$ Hence $(\cT,\Sigma F) = 0$ so $F \in \{t \hspace{1.5mm} | \hspace{1.5mm} (\cT, \Sigma t) = 0 \} = \cT,$ which is a contradiction. So we conclude that $p=m$, and symmetrically, $q=m.$ Observe that in this case, $\fT'$ is in fact a maximal, non-crossing set of finite arcs. For if a finite arc $\fa$ crosses no arc in $\fT'$, then $\fa$ doesn't overarc $m$, because that's where $\fT'$ has a fountain. Hence $\fa$ also doesn't cross $(m,\infty).$ Therefore $\fa \in \fT$ so $\fa \in \fT'$.

It is impossible for $\fT$ to contain two or more infinite arcs. For if it did, they would correspond to different hocolims $\sE, \sE' \in \cT$, contradicting Proposition \ref{pro:wct}. \\ \\
 $[\Leftarrow]$ There are two cases. \\
Case (1). Suppose that $\fT$ is a set of finite arcs which is maximal, non-crossing and locally finite in the sense of \citep[Section 3]{J}. Then we show that $\cT = \{t \hspace{1.5mm} | \hspace{1.5mm} (t,\Sigma \cT) = 0 \}$ and $\cT = \{t \hspace{1.5mm} | \hspace{1.5mm} (\Sigma \cT,t) = 0 \}$. In fact it is enough just to consider the first of these equations,
\begin{equation} \label{equ:ctis}
\cT = \{t \hspace{1.5mm} | \hspace{1.5mm} (t,\Sigma \cT) = 0 \},
\end{equation}
as the second equation is symmetric. There are two inclusions to show in order to establish Equation (\ref{equ:ctis}). In each case it is enough to consider indecomposable objects.

The inclusion $\cT \subseteq \{t \hspace{1.5mm} | \hspace{1.5mm} (t,\Sigma \cT) = 0 \}$ follows because the arcs in $\fT$ do not cross.

For the inclusion $\cT \supseteq \{t \hspace{1.5mm} | \hspace{1.5mm} (t,\Sigma \cT) = 0 \}$, consider an indecomposable object $s \in \{t \hspace{1.5mm} | \hspace{1.5mm} (t,\Sigma \cT) = 0 \}$. We aim to show that $s \in \cT.$ There are two possibilities: either $s \in \cD$ or $s = \sE_h,$ for some $h \in \BZ$. If $s \in \cD,$ then $s \in \{t \hspace{1.5mm} | \hspace{1.5mm} (t,\Sigma \cT) = 0 \}$ means that the arc of $s$ crosses no arc in $\fT.$ Then the arc of $s$ is in $\fT$ and hence $s \in \cT.$ If, on the other hand, $s=\sE_h$, then by Corollary \ref{cor:locallyfinitesecondclaim}, $s \in \{t \hspace{1.5mm} | \hspace{1.5mm} (t,\Sigma \cT) = 0 \}$ cannot happen, because there will always be an arc in $\fT$ which crosses the arc associated with $\sE_h.$ \\
Case (2). Suppose that $\fT = \{(m,\infty)\} \cup \fT'$ where $\fT'$ is a set of finite arcs which is maximal non-crossing and has a fountain at $m$. Like before, we show that $\cT = \{t \hspace{1.5mm} | \hspace{1.5mm} (t,\Sigma \cT) = 0 \}$ and $\cT = \{t \hspace{1.5mm} | \hspace{1.5mm} (\cT,\Sigma t) = 0 \}$, and again it is enough just to consider the first of these equations as the second equation is symmetric. There are two inclusions to show in order to establish Equation (\ref{equ:ctis}). Again in each case it is enough to consider indecomposable objects.

The inclusion $\cT \subseteq \{t \hspace{1.5mm} | \hspace{1.5mm} (t,\Sigma \cT) = 0 \}$ follows because the arcs in $\fT$ do not cross.

For the inclusion $\cT \supseteq \{t \hspace{1.5mm} | \hspace{1.5mm} (t,\Sigma \cT) = 0 \}$, consider an indecomposable object $s \in \{t \hspace{1.5mm} | \hspace{1.5mm} (t,\Sigma \cT) = 0 \}$. We aim to show that $s \in \cT.$ There are two possibilities: either $s \in \cD$ or $s = \sE_h,$ for some $h \in \BZ$. If $s \in \cD,$ then $s \in \{t \hspace{1.5mm} | \hspace{1.5mm} (t,\Sigma \cT) = 0 \}$ means that the arc of $s$ crosses no arc in $\fT.$ Then, since $\fT'$ is maximal non-crossing, the arc of $s$ is in $\fT'$ and hence $s \in \cT.$ If $s=\sE_h$, then $s \in \{t \hspace{1.5mm} | \hspace{1.5mm} (t,\Sigma \cT)=0\}$ means that the arc of $\sE_h$ crosses no arc in $\fT'$ which has a fountain at $m$. Hence the arc of $\sE_h$ must be $(m,\infty),$ i.e., $h=-m-2,$ so $s=\sE_h \in \cT.$
\end{proof}

\section{Cluster tilting subcategories}
\label{sec:cts}

Leading on from the last section, this section concludes with a theorem stating when a subcategory of $\Dbar$ is cluster tilting. \noindent Right now we aim to show that $\cT$ is functorially finite if and only if $\fT$ has a fountain.

\begin{Definition}\label{ff}
Let $\cT$ be a subcategory of $\Dbar$. We say that $\cT$ is {\textit{right-approximating}} if for any $d \in \Dbar$ there exists a $\cT$-object $t$, and a right-$\cT$-approximation $\tau : t \rightarrow d.$ This means that for any $\tau' : t' \rightarrow d,$ there exists a factorisation
\begin{equation} \begin{aligned} \label{equ:rightapprox}
\xymatrix{
& t' \ar@{-->}[dl]_{\exists} \ar[d]^{\tau'} \\
t \ar[r]_{\tau} & d
}.
\end{aligned} \end{equation}
Analogously, $\cT$ is left-approximating if for any $\Dbar$-object $d$ there exists a left-approximation $\tau : d \rightarrow t.$ This means that for any $\cT$-object $t'$ and a morphism $\tau' : d \rightarrow t',$ there exists a factorisation
\begin{equation} \begin{aligned} \label{equ:leftapprox}
\xymatrix{
& t' \\
t \ar@{-->}[ur]^{\exists} & d \ar[l]^{\tau} \ar[u]_{\tau'}
}.
\end{aligned} \end{equation}
We say that $\cT$ is {\textit{functorially finite}} if and only if it is both right-approximating and left-approximating. \demo
\end{Definition}

\begin{Definition}\label{almostrightdef}
We say that a property holds for almost all indecomposables in $\cT$ if it holds for all but finitely many indecomposables.
An {\textit{almost-right-}}$\cT${\textit{-approximation}} of $d \in \Dbar$ is a morphism $\tau : t \rightarrow d$ such that for almost all indecomposables $t'$ in $\cT$, each morphism $\tau' : t' \rightarrow d$ factors through $\tau$. An almost-left-$\cT$-approximation is defined analogously. \demo
\end{Definition}

\begin{Lemma}\label{almostrightlem}
An almost-right-$\cT$-approximation of a $\Dbar$-object $d$ exists if and only if a right-$\cT$-approximation of $d$ exists (and similarly for the case of left-approximations).
\end{Lemma}

\begin{proof}
Given an almost-right-$\cT$-approximation $\widetilde{\tau} : \widetilde{t} \rightarrow d$, let $t_1,\dots,t_j$ be the finitely many indecomposables in $\cT$ for which morphisms to $d$ don't necessarily factor through $\widetilde{\tau}$. We get a right-$\cT$-approximation 
$$t = \widetilde{t}\oplus s_{1,1}\oplus \cdots \oplus s_{1,n_1} \oplus s_{2,1} \oplus \cdots \oplus s_{2,n_2} \oplus \cdots \oplus s_{j,1} \oplus \cdots \oplus s_{j,n_j}$$ where
$$\dim Hom(t_x,d) = n_x.$$ 
The opposite implication is seen instantly: by definition, right-$\cT$-approximations are themselves almost-right-$\cT$-approximations. The proof for left-approximations is done analogously.
\end{proof}

\begin{Remark}\label{rem:choice} 
In Notation \ref{not:choice}, we write hocolimits in terms of indecomposable objects; that is, $\sE_n$ is written as  $\hocolim_i(\Sigma^{n-i}X_i).$ We draw attention to the fact that we may use coordinate-pair notation, and may write, for example, $\hocolim(-n-2,\text{---})$ instead of  $\hocolim_i(\Sigma^{n-i}X_i).$ We like to be flexible and, to this end, may continue to write $\sE_n$ instead of $(-n-2,\infty).$ Before proving that $\cT$ is functorially finite if and only if $\fT$ has a fountain, we prove two important factorisation lemmas. \demo
\end{Remark}

\begin{Lemma}\label{lem:fact1}
Let $k \in \BN$ and let $\sE_{n+k}$ $=$ $\hocolim_i(\Sigma^{n+k-i}X_i).$ Let $z_1, z_2 \in \ind\cD$ where
$$z_1 = (y,-n) : y<-n-k-2,$$
$$z_2 = (x,-n+p) : y \leq x \leq -n-k-2 \hspace{2mm} \text{and} \hspace{2mm}  p \geq 0.$$ 
Then any nonzero morphism 
$$\tau' : z_1 \rightarrow \sE_{n+k}$$
factorises as
$$\xymatrix{
z_1 \ar[r]^{g} & z_2 \ar[r]^{\tau} & \sE_{n+k}.
}$$
\end{Lemma}

\begin{proof}
First, let us clarify the situation. The objects $z_1$ and $z_2$ are defined so as to satisfy $z_1, z_2 \in \BW(\Sigma^nX_0)$ and $\hmD(z_1,z_2) \neq 0.$ This is illustrated below.
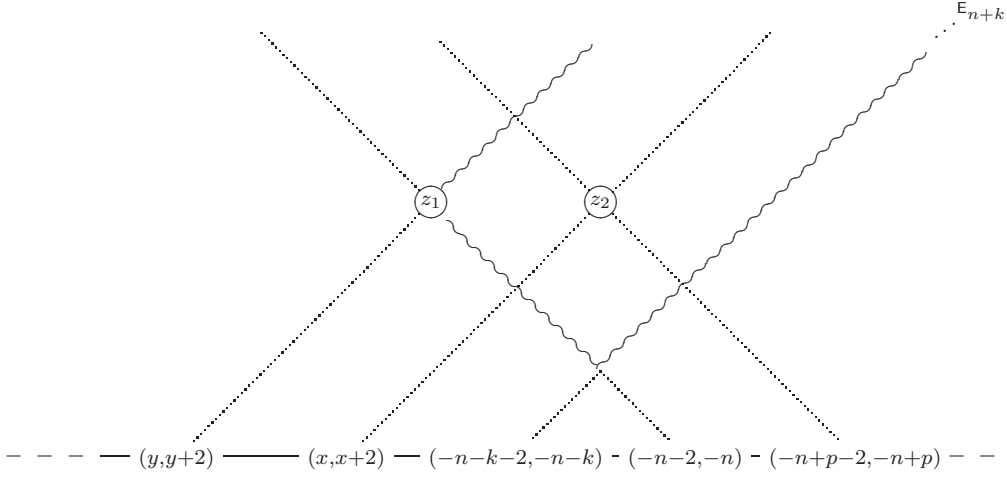
\[ \def\objectstyle{\scriptstyle}
  \xymatrix @-5.4pc @! {
    & *{} & & *{} & & & & & *{} & *{} & *{} & \hspace{8mm} {\iddots}^{\sE_{n+k}}   & \\
    & & & & & & & & & & & & \\
    & & & & & *+[o][F]{z_1} \ar@{~}[uurr] \ar@{.}[uull] \ar@{~}[ddrr] & & *+[o][F]{z_2} \ar@{.}[uull] \ar@{.}[uurr] & & & & & \\
    & & & & & *{} & & & & & & & \\
    & & & & & & & *{} \ar@{~}[uuuurrrr] & & & & & \\
    *{} \ar@{--}[r] & \ar@{-}[r] & (y,y+2) \ar@{-}[rr] \ar@{.}[uuurrr] & & (x,x+2) \ar@{-}[rr] \ar@{.}[uuurrr]  & & (-n-k-2,-n-k) \ar@{.}[ur]  \ar@{-}[rr] & & (-n-2,-n) \ar@{.}[ul] \ar@{-}[rr] & & (-n+p-2,-n+p) \ar@{--}[rr] \ar@{.}[uuulll] & & *{} \\
           }
\]
Here, $z_2$ can be located anywhere within the region bounded by wavy lines. Now, the following diagram is natural in the variable $z.$
%this also needs a reference
\begin{equation} \begin{aligned} \label{equ:smallcommdiag} 
\xymatrix{
\colim_i(z,\Sigma^{n+k-i}X_i) \ar[rr]^{\cong} && (z,\sE_{n+k}) \\
& (z,\Sigma^{n+k-j}X_j) \ar[ur] \ar[ul] }
\end{aligned} \end{equation}
Hence, we may draw the following commutative diagram.
\begin{equation} \begin{aligned} \label{equ:bigcommdiag} 
\xymatrix{
\colim_i(z_2,\Sigma^{n+k-i}X_i) \ar[rr]^{\cong} \ar[dd] && (z_2,\sE_{n+k}) \ar[dd] \\
& (z_2,\Sigma^{n+k-j}X_j) \ar[ur] \ar[ul] \ar[dd] \\
\colim_i(z_1,\Sigma^{n+k-i}X_i) \ar[rr]^{\cong} & \ar[d] & (z_1,\sE_{n+k}) \\
& (z_1,\Sigma^{n+k-j}X_j) \ar[ur] \ar[ul] }
\end{aligned} \end{equation}
The object $\Sigma^{n+k-j}X_j$ has coordinate $(-n-k-2,-n-k+j).$

Now, let $\tau' \in (z_1,\sE_{n+k}).$ For $j$ sufficiently big there is a $\widetilde{\tau} \in (z_1,\Sigma^{n+k-j}X_j)$ which maps to $\tau'$ (found by diagram-chasing in Figure (\ref{equ:bigcommdiag})). That is, $\tau' = f \circ \widetilde{\tau}$ where $f : \Sigma^{n+k-j}X_j \rightarrow \sE_{n+k}$ is the canonical morphism. By Lemma 2.5 of \citep{J}, 
$$\xymatrix{z_1 \ar[r]^-{\widetilde{\tau}} & \Sigma^{n+k-j}X_j}$$
factors through 
$$\xymatrix{z_1 \ar[r]^{g} & z_2.}$$
Let the associated map from $z_2$ to $\Sigma^{n+k-j}X_j$ be called $\tau.$ Pictorially, we have the following.
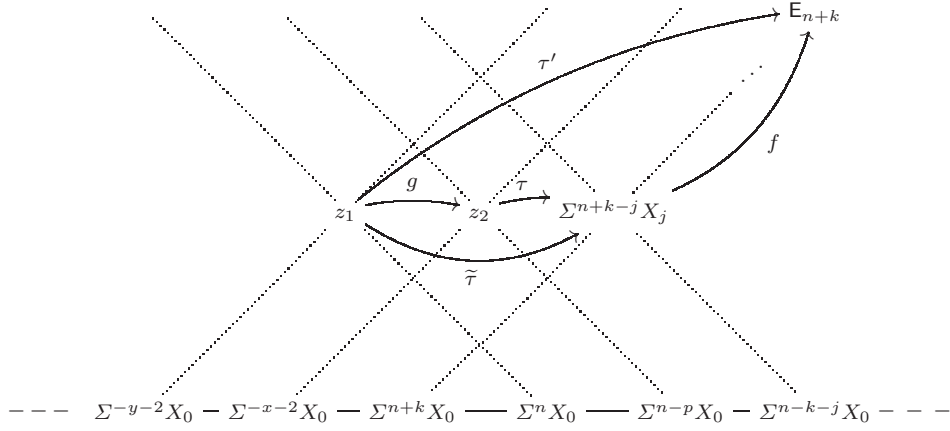
\[ \def\objectstyle{\scriptstyle}
  \xymatrix @-4pc @! {
    & *{} & & *{} & & & & & *{} & *{} & *{} & & {{\sE_{n+k}}} \\
    & & & & & & & & & & & {\hspace{1mm} \iddots} & \\
    & & & & & & & & & & & & \\
    & & & & & {z_1} \ar@{.}[uuurrr] \ar@{.}[uuulll] \ar@/^0.30pc/[rr]^-{g} \ar@{.}[ddrr] \ar@/_1.60pc/[rrrr]_-{\widetilde{\tau}} \ar@/^1.20pc/[uuurrrrrrr]^-{\tau'} & & {z_2} \ar@{.}[uuulll] \ar@{.}[uuurrr] \ar@/^0.40pc/[rr]^-{\tau} & & {\Sigma^{n+k-j}X_j} \ar@{.}[uuulll] \ar@{.}[uurr] \ar@/_1.40pc/[uuurrr]_-{f} & & & \\
    & & & & & *{} & & & & & & & \\
    & & & & & & & *{} \ar@{.}[uurr] & & & & & \\
    *{} \ar@{--}[r] & \ar@{-}[r] & \Sigma^{-y-2}X_0 \ar@{-}[rr] \ar@{.}[uuurrr] & & \Sigma^{-x-2}X_0 \ar@{-}[rr] \ar@{.}[uuurrr]  & & \Sigma^{n+k}X_0 \ar@{.}[ur]  \ar@{-}[rr] & & \Sigma^nX_0 \ar@{.}[ul] \ar@{-}[rr] & & \Sigma^{n-p}X_0 \ar@{-}[rr] \ar@{.}[uuulll] & & \Sigma^{n-k-j}X_0 \ar@{.}[uuulll] \ar@{-}[r] \ar@{--}[rr] & & *{} \\
           }
\]

\noindent The commutative diagram (\ref{equ:bigcommdiag}) contains the following subdiagram.
$$\xymatrix{
(z_2,\sE_{n+k}) \ar[rrr]^-{(g,\sE_{n+k})} & & & (z_1,\sE_{n+k}) \\ \\
(z_2,\Sigma^{n+k-j}X_j) \ar[uu]^-{(z_2,f)} \ar[rrr]_-{(g,\Sigma^{n+k-j}X_j)} & & & (z_1,\Sigma^{n+k-j}X_j) \ar[uu]_-{(z_1,f)}
}
$$
Diagram-chasing yields the following.
$$\xymatrix{
f \circ \tau \ar@{|-{>}}[rrr]^-{(g,\sE_{n+k})} & & & f \circ \tau \circ g = \tau' \\ \\
\tau \ar@{|-{>}}[uu]^-{(z_2,f)} \ar@{|-{>}}[rrr]_-{(g,\Sigma^{n+k-j}X_j)} & & & \widetilde{\tau} \ar@{|-{>}}[uu]_-{(z_1,f)}
}
$$
Hence, $\tau'$ has a preimage in $(z_2,\sE_{n+k}),$ namely $f \circ \tau.$ Therefore $\tau'$ factors through $g : z_1 \rightarrow z_2,$ as required.
\end{proof}

\begin{Definition} \label{def:idnminusshift}
Let 
$$\xymatrix{ y_i \ar[rr]^-{\upsilon_i} && y_{i+1} \ar[rr]^-{\upsilon_{i+1}} && y_{i+2} \ar[rr]^-{\upsilon_{i+2}} && \cdots}$$
be a direct system of finite indecomposable objects in $\Dbar$. Then we define the map
$$\xymatrix{ \displaystyle\coprod_{k=i}^{\infty}y_k \ar[rr]^-{\text{id}-\text{shift}} && \displaystyle\coprod_{k=i}^{\infty}y_k}$$
by
$$\text{id}-\text{shift} = \begin{pmatrix} 1 & 0 & 0 & \cdots \\ -\upsilon_i & 1 & 0 & \cdots \\ 0 & -\upsilon_{i+1} & 1 & \cdots \\ 0 & 0 & -\upsilon_{i+2} & \cdots \\ \vdots & \vdots & \vdots & \ddots  \end{pmatrix}.$$ 
Note that the map ``$\text{id}-\text{shift}$'' is dependent on the morphisms in its associated direct system, which is usually made clear by context. \demo
\end{Definition}

\begin{Lemma}\label{lem:fact2}
Let $k<0 \leq i$ be integers. Let $z_i$ be the indecomposable with coordinate $(-n-2,-n-k+i)$ in $\cD,$ and let $\sE_{n+k}$ $=$ $\hocolim_j(\Sigma^{n+k-j}X_j).$ Then any nonzero morphism $\tau' : z_i \rightarrow \sE_{n+k}$ factorises as $\xymatrix{
z_i \ar[r]^-{g_i} & \sE_n \ar[r]^-{\tau} & \sE_{n+k}
}
$ where $\sE_n$ is the hocolimit associated with the slice $(-n-2,\text{---}).$
\end{Lemma}
\begin{proof}
Let $j>i$ be an integer. Then $z_i, z_j \in \BW(\Sigma^{n+k}X_0) \cap \BW(\Sigma^{n}X_0),$ with $z_i, z_j$ sitting on the slice $(-n-2,\text{---}),$ and $\hmD(z_i,z_j)\neq 0.$ Let $y_i$ be the ``corresponding'' indecomposable sat on the slice $(-n-k-2,\text{---});$ that is, $y_i = (-n-k-2,-n-k+i).$ The following picture shows all of this. 
$$ \def\objectstyle{\scriptstyle}
 
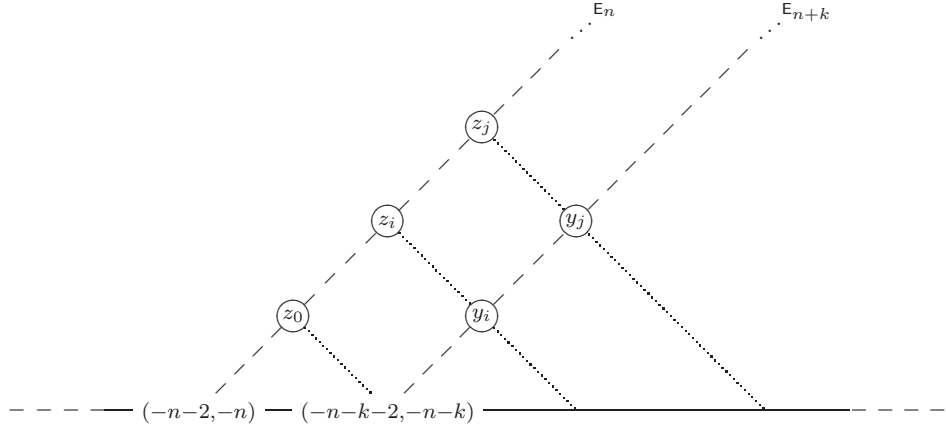
\centering  \xymatrix @-5.1pc @! {
& & & & & &\hspace{4.5mm} {\iddots}^{\sE_{n}}& & \hspace{8mm} {\iddots}^{\sE_{n+k}}  & &    \\    
& & & & & *+[o][F]{z_j} \ar@{--}[ur] \ar@{.}[dr] & & & & &    \\
    & & & & *+[o][F]{z_i} \ar@{--}[ur] \ar@{.}[dr] & & *+[o][F]{y_j} \ar@{.}[ddrr] \ar@{.}[ddrr] \ar@{--}[uurr] & & & &  \\
    & & &*+[o][F]{z_0} \ar@{--}[ur] & & *+[o][F]{y_i} \ar@{--}[ur] \ar@{.}[dr] & & & & &    \\
   *{} \ar@{--}[r] & *{} \ar@{-}[r] & (-n-2,-n)\ar@{-}[rr] \ar@{--}[ur] & & (-n-k-2,-n-k) \ar@{-}[rrrrr] \ar@{--}[ur] \ar@{.}[ul] & & & *{} \ar@{--}[rrr] & & &  \\
           }
$$
Suppose $j=i+1.$ Then, by Lemma \ref{lem:fact1}, $\tau' : z_i \rightarrow \sE_{n+k}$ can be factored through $z_{i+1}, y_i,$ and $y_{i+1},$ to form a commutative diagram illustrated below.
$$ \def\objectstyle{\scriptstyle}
  
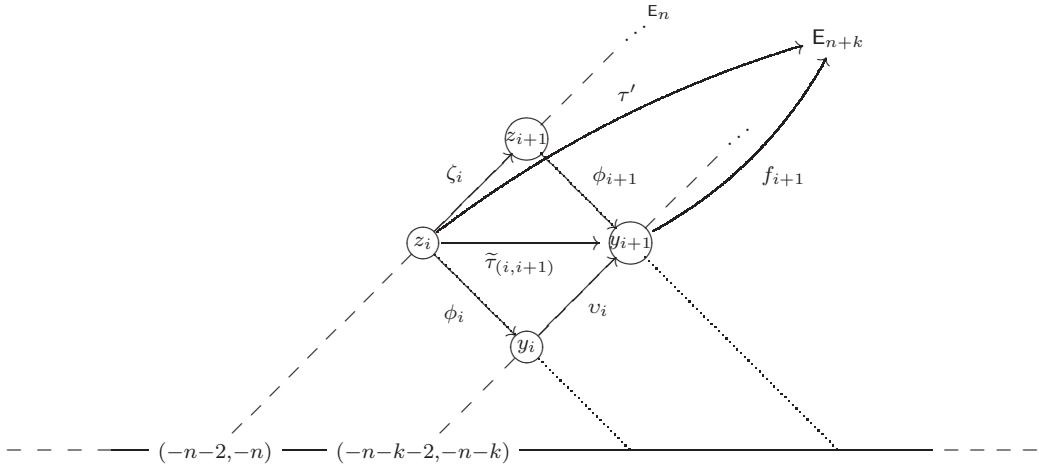
\xymatrix @-4.8pc @! {
& & & & & &\hspace{4.5mm} {\iddots}^{\sE_{n}}& & {\sE_{n+k}}  & &  \\    
& & & & & *+[o][F]{z_{i+1}} \ar@{--}[ur] \ar@{.}[dr] \ar[dr]^>>>>>>{\phi_{i+1}} & & \hspace{1mm} {\iddots} & & &  \\
    & & & & *+[o][F]{z_i} \ar@{--}[ur]  \ar[ur]^-{\zeta_i} \ar[dr]_-{\phi_i}  \ar@{.}[dr] \ar[rr]_-{\widetilde{\tau}_{(i,i+1)}} \ar@/^0.75pc/[uurrrr]^>>>>>>>>>>>>{\tau'} & & *+[o][F]{y_{i+1}} \ar@{.}[ddrr] \ar@{.}[ddrr] \ar@{--}[ur] \ar@/^-0.8pc/[uurr]_{f_{i+1}} & & & & \\
    & & & & & *+[o][F]{y_i} \ar@{--}[ur] \ar@{.}[dr] \ar[ur]_{\upsilon_i} & & & & &  \\
   *{} \ar@{--}[r] & *{} \ar@{-}[r] & (-n-2,-n)\ar@{-}[rr] \ar@{--}[uurr] & & (-n-k-2,-n-k) \ar@{-}[rrrrr] \ar@{--}[ur] & & & *{} \ar@{--}[rrr] & & &   \\
           }
$$
Note that $\widetilde{\tau}_{(i,i+1)} = \phi_{i+1} \circ \zeta_i = \upsilon_{i} \circ \phi_i.$ This method induces the following ladder.
\begin{equation} \begin{aligned} \label{equ:zyladder}
\xymatrix{
z_i \ar[rr]^-{\zeta_i} \ar[dd]_-{\phi_i} \ar@/^0.65pc/[ddrr]_-{\widetilde{\tau}_{(i,i+1)}} \ar@/^1.3pc/[ddrrrr]^>>>>>>>>>>{\widetilde{\tau}_{(i,i+2)}} && z_{i+1} \ar[rr]^-{\zeta_{i+1}} \ar[dd]^>>>>>>>{\phi_{i+1}} && z_{i+2} \ar[rr] \ar[dd]^-{\phi_{i+2}} && \cdots \\ \\
y_i \ar[rr]_-{\upsilon_{i}} && y_{i+1} \ar[rr]_-{\upsilon_{i+1}} && y_{i+2} \ar[rr] \ar[rr] && \cdots
}
\end{aligned} \end{equation}
Recall that hocolimits are defined in \citep{BN} on page 209. Recall Definition \ref{def:idnminusshift}. Consider the direct system
$$\xymatrix{ y_i \ar[rr]^-{\upsilon_i} && y_{i+1} \ar[rr]^-{\upsilon_{i+1}} && y_{i+2} \ar[rr]^-{\upsilon_{i+2}} && \cdots}$$
and its associated ``$\text{id}-\text{shift}$'' map
$$\xymatrix{ \displaystyle\coprod_{k=i}^{\infty}y_k \ar[rr]^-{\text{id}-\text{shift}} && \displaystyle\coprod_{k=i}^{\infty}y_k}.$$
Then in fact the hocolimit of this direct system is defined as the mapping cone of the ``$\text{id}-\text{shift}$'' map, which can be found by completing to a distinguished triangle the following diagram.
\begin{equation} \begin{aligned} \label{equ:idsh} \xymatrix{ \displaystyle\coprod_{k=i}^{\infty}y_k \ar[rr]^-{\text{id}-\text{shift}} && \displaystyle\coprod_{k=i}^{\infty}y_k  \ar@{-->}[rr]^-{\Phi} && \hocolim y_i \ar[r] &} \end{aligned} \end{equation}
There are coproduct inclusions
$$\iota_r : y_r \rightarrow \displaystyle\coprod_{k=i}^{\infty}y_k$$
for $r \geq i$. Let $r \geq i$ be an integer.
\begin{equation} \begin{aligned} \label{equ:idsh2} \xymatrix{ \displaystyle\coprod_{k=i}^{\infty}y_k \ar[rr]^-{\text{id}-\text{shift}} && \displaystyle\coprod_{k=i}^{\infty}y_k  \ar@{-->}[rr]^-{\Phi} && \hocolim y_i \ar[r] & \\ \\
y_{r} \ar@{^{(}->}[uu]^-{\iota_{r}} && y_{r} \ar@{^{(}->}[uu]^-{\iota_{r}} \ar[uurr]_-{f_{r}} } \end{aligned} \end{equation}
Now, because $\Phi \circ (\text{id}-\text{shift}) = 0,$ we see that $f_{r+1} \circ \upsilon_r = f_r.$ Hence the ladder in Figure (\ref{equ:zyladder}) can be extended in the following way. \\ \\
\begin{equation} \begin{aligned} \label{equ:ladderwithhocolims}
\xymatrix{
z_i \ar[rr]^-{\zeta_i} \ar[dd]_-{\phi_i} \ar@/^0.65pc/[ddrr]_-{\widetilde{\tau}_{(i,i+1)}} \ar@/^1.3pc/[ddrrrr]^>>>>>>>>>>{\widetilde{\tau}_{(i,i+2)}} \ar@/^4.5pc/[rrrrrrr]^<<<<<<<<<<<<{g_i} && z_{i+1} \ar[rr]^-{\zeta_{i+1}} \ar[dd]^>>>>>>>{\phi_{i+1}} \ar@/^3pc/[rrrrr]^<<<<<<<<<{g_{i+1}}  && z_{i+2} \ar[rr] \ar[dd]^-{\phi_{i+2}} \ar@/^1.5pc/[rrr]^<<<<<{g_{i+2}} && \cdots & \sE_n \\ \\
y_i \ar[rr]_-{\upsilon_{i}} \ar@/_4.5pc/[rrrrrrr]_<<<<<<<<<<<<{f_i} && y_{i+1} \ar[rr]_-{\upsilon_{i+1}} \ar@/_3pc/[rrrrr]_<<<<<<<<<{f_{i+1}} && y_{i+2} \ar[rr] \ar[rr] \ar@/_1.5pc/[rrr]_<<<<<{f_{i+2}} && \cdots  & \sE_{n+k}
}
\end{aligned} \end{equation} \\ \\
There are two distinguished triangles:
\begin{equation} \begin{aligned} \label{equ:idsh3} \xymatrix{ \displaystyle\coprod_{k=i}^{\infty}y_k \ar[rr]^-{\text{id}-\text{shift}} && \displaystyle\coprod_{k=i}^{\infty}y_k  \ar@{-->}[rr]^-{} && \hocolim y_i \ar[rr] && \Sigma\displaystyle\coprod_{k=i}^{\infty}y_k} \end{aligned} \end{equation}
and
\begin{equation} \begin{aligned} \label{equ:idsh4} \xymatrix{ \displaystyle\coprod_{k=i}^{\infty}z_k \ar[rr]^-{\text{id}-\text{shift}} && \displaystyle\coprod_{k=i}^{\infty}z_k  \ar@{-->}[rr]^-{} && \hocolim z_i \ar[rr] && \Sigma\displaystyle\coprod_{k=i}^{\infty}z_k,} \end{aligned} \end{equation}
which, by axiom (TR3) of triangulated categories (see, for example, \citep[Definition 10.2.1]{W}) can be connected with morphisms in the following way, creating a commutative diagram with rows distinguished triangles.
\begin{equation} \begin{aligned} \label{equ:idshladder} \xymatrix{ \displaystyle\coprod_{k=i}^{\infty}z_k \ar[rr]^-{\text{id}-\text{shift}} \ar[dd]_-{\coprod\phi_k} && \displaystyle\coprod_{k=i}^{\infty}z_k  \ar@{-->}[rr]^-{} \ar[dd]_-{\coprod\phi_k} && \hocolim z_i \ar[rr] \ar@{-->}[dd]_-{\tau} && \Sigma\displaystyle\coprod_{k=i}^{\infty}z_k \ar[dd]_-{\Sigma\coprod\phi_k} \\ \\
\displaystyle\coprod_{k=i}^{\infty}y_k \ar[rr]^-{\text{id}-\text{shift}} && \displaystyle\coprod_{k=i}^{\infty}y_k  \ar@{-->}[rr]^-{} && \hocolim y_i \ar[rr] && \Sigma\displaystyle\coprod_{k=i}^{\infty}y_k} \end{aligned} \end{equation}
Hence, Figure (\ref{equ:zyladder}) can be completed with a morphism from $\sE_n$ to $\sE_{n+k}.$\\
\begin{equation} \begin{aligned} \label{equ:ladderwithhocolims2}
\xymatrix{
z_i \ar[rr]^-{\zeta_i} \ar[dd]_-{\phi_i} \ar@/^0.65pc/[ddrr]_-{\widetilde{\tau}_{(i,i+1)}} \ar@/^1.3pc/[ddrrrr]^>>>>>>>>>>{\widetilde{\tau}_{(i,i+2)}} \ar@/^4.5pc/[rrrrrrr]^<<<<<<<<<<<<{g_i} && z_{i+1} \ar[rr]^-{\zeta_{i+1}} \ar[dd]^>>>>>>>{\phi_{i+1}} \ar@/^3pc/[rrrrr]^<<<<<<<<<{g_{i+1}}  && z_{i+2} \ar[rr] \ar[dd]^-{\phi_{i+2}} \ar@/^1.5pc/[rrr]^<<<<<{g_{i+2}} && \cdots & \sE_n \ar[dd]_-{\tau} \\ \\
y_i \ar[rr]_-{\upsilon_{i}} \ar@/_4.5pc/[rrrrrrr]_<<<<<<<<<<<<{f_i} && y_{i+1} \ar[rr]_-{\upsilon_{i+1}} \ar@/_3pc/[rrrrr]_<<<<<<<<<{f_{i+1}} && y_{i+2} \ar[rr] \ar[rr] \ar@/_1.5pc/[rrr]_<<<<<{f_{i+2}} && \cdots  & \sE_{n+k}
}
\end{aligned} \end{equation} \\ \\
Thus, $\tau' = f_{i+1} \circ \widetilde{\tau}_{(i,i+1)} = \tau \circ g_i,$ which means that $\tau' : z_i \rightarrow \sE_{n+k}$ factors through $\tau : \sE_n \rightarrow \sE_{n+k},$ as required.
\end{proof}

\noindent Now we are in a position to prove that $\cT$ is functorially finite if and only if $\fT$ has a fountain.

\begin{Proposition}
Let $\cT$ be a weakly cluster tilting subcategory of the form in Theorem \ref{pro:maximal}, which contains the hocolimit $\sE_n,$ with corresponding set of arcs $\fT.$ Then $\cT$ is functorially finite.
\end{Proposition}

\begin{proof}
We begin by proving that if $\fT$ has a fountain at $-n-2$, then $\Dbar$ has almost-right-$\cT$-approximations. If $\fT$ has a fountain at $-n-2,$ then in the Auslander-Reiten quiver, the slice $(-n-2,\text{---})$ and the coslice $(\text{---},-n-2)$ have an infinite number of $\cT$-objects on them, and no other slice or coslice has this property. Suppose we seek an almost-right-$\cT$-approximation of $d,$ where $d$ sits on the slice $(-n-3,\text{---}),$ or is to the right of it in the quiver, as shown below.
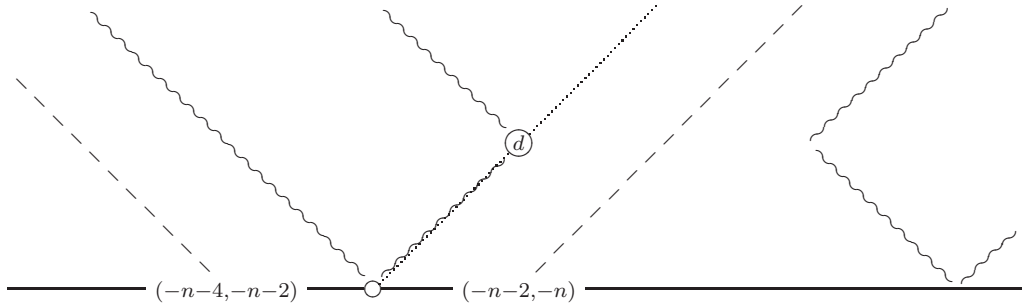
\[ \centering \def\objectstyle{\scriptstyle}
   \xymatrix @-4.8pc @! {
   & & & & & & & & & & & & &  \\
   & & & & & & & & & & & & &  \\
   & & & & & & & *+[o][F]{d} \ar@{.}[uurr] \ar@{~}[uull] & & & & *{} \ar@{~}[uurr] \ar@{~}[ddrr] & &  \\
   & & & & & & & & & & & & & &   \\
   *{} \ar@{-}[r]   & *{} \ar@{-}[rr] & & (-n-4,-n-2) \ar@{-}[rr] \ar@{--}[uuulll] & &  *+[o][F]{} \ar@{-}[rr] \ar@{.}[uurr] \ar@{~}[uurr] \ar@{~}[uuuullll] & & (-n-2,-n) \ar@{-}[rrrrrrr] \ar@{--}[uuuurrrr] & & & & & & *{} \ar@{~}[ur] &  \\
           }
\]

The region which has nonzero maps to $d,$ illustrated with wavy lines, does not contain an infinite number of $\cT$-objects, because it intersects only finitely many slices and coslices which all have only finitely many $\cT$-indecomposables. This means that if $d$ is on $(-n-3,\text{---}),$ or is to the right of it in the quiver, then only finitely many objects in $\ind\cT$ have nonzero morphisms to $d$. Hence, $0 \rightarrow d$ is an almost-right-$\cT$-approximation of $d.$ 

If $d$ is on the coslice $(\text{---},-n-3),$ or is left of it on the quiver, then $d$ also has a right $\cT$-approximation, for the same reasons as above. This is shown below.
$$ \centering \def\objectstyle{\scriptstyle}
  
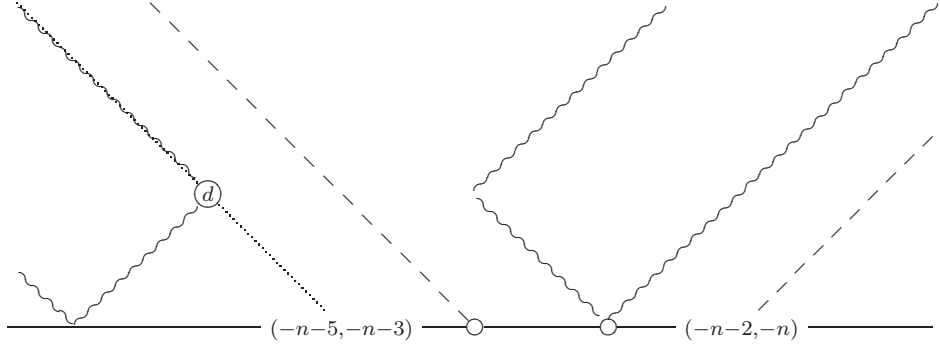
\xymatrix @-5.0pc @! {
   & & & & & & & & & & & & & & & & \\
   & & & & & & & & & & & & & & & & \\
   & & & & & & & & & & & & & & & & \\
   & & & &  *+[o][F]{d} \ar@{.}[uuulll] \ar@{~}[uuulll] \ar@{~}[ddll] & & & & *{} \ar@{~}[uuurrr] & & & & & \\
   & & & & & & & & & & & & & & \\
   &*{} \ar@{-}[rrrrr] & *{} \ar@{~}[ul] & & & & (-n-5,-n-3) \ar@{.}[uull] \ar@{-}[rr] & & *+[o][F]{} \ar@{--}[uuuuulllll] \ar@{-}[rr] & & *+[o][F]{} \ar@{-}[rr] \ar@{~}[uull] \ar@{~}[uuuuurrrrr] & & (-n-2,-n) \ar@{--}[uuurrr] \ar@{-}[rrr] & & & \\
           }
$$
\noindent Now suppose that $d \in \BW(\Sigma^{n+2}X_0).$ This is illustrated below.
$$ \centering \def\objectstyle{\scriptstyle}
  
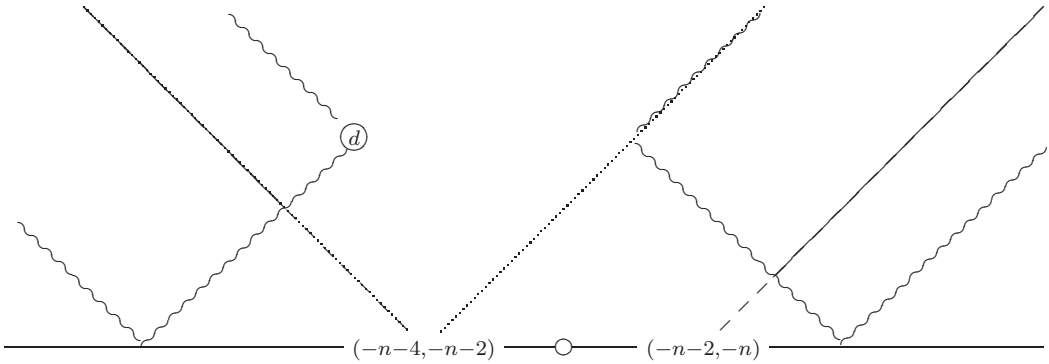
\xymatrix @-4.9pc @! {
  & & & & & & & & & & & & & & & \\
  & & & & & & & & & & & & & & & \\
  & & & & & *+[o][F]{d} \ar@{~}[dddlll] \ar@{~}[uull] & & & & *{} \ar@{~}[dddrrr] \ar@{~}[uurr] & & & & & & \\
  & & & & *{} \ar@{-}[uuulll] & & & & & & & & & & & \\
  & & & & & & & & & & & *{} \ar@{-}[uuuurrrr] & &  & & \\
  *{} \ar@{-}[rrrrrr] & & *{} \ar@{~}[uull] & & & & (-n-4,-n-2) \ar@{.}[uuuuurrrrr] \ar@{--}[uuuuulllll] \ar@{.}[uuuuulllll]  \ar@{-}[rr] & & *+[o][F]{} \ar@{-}[rr] & & (-n-2,-n) \ar@{--}[uuuuurrrrr] \ar@{-}[rrrrr] & & \ar@{~}[uuurrr] & & & \\
}
$$
The wavy lines indicate the region with nonzero maps to $d.$ The dotted lines illustrate $\BW(\Sigma^{n+2}X_0).$ The dashed lines are the slice/coslice which contain an infinite number of $\cT$-objects. The solid lines illustrate the part of the slice/coslice which have an infinite number of $\cT$-objects which also map to $d.$ Pick $\widetilde{t} \in \cT \cap (\text{---},-n-2)$ and pick a nonzero morphism $\widetilde{t} \rightarrow d.$ Then $\widetilde{t} = (-n-k,-n-2),$ where $k \geq 4.$  By Lemma 2.5 of \citep{J}, each $t_1 \rightarrow d,$ with $t_1 = (-n-s,-n-2)$ for $s \geq k,$ factors through $\widetilde{t} \rightarrow d.$ This is shown below.
$$ \centering \def\objectstyle{\scriptstyle}
  
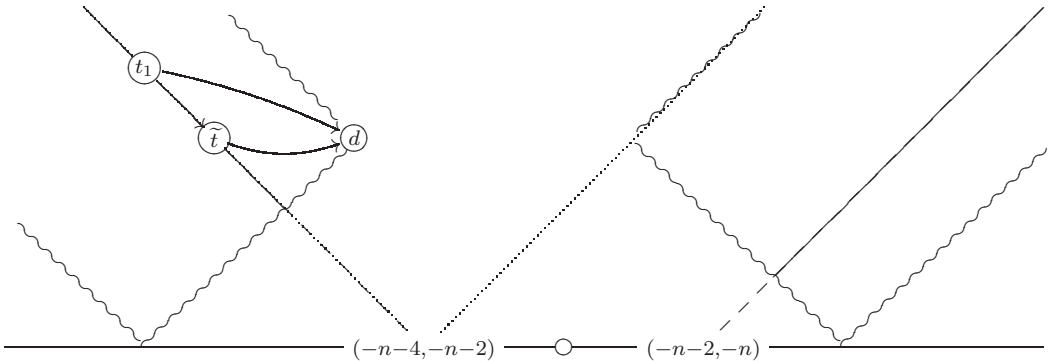
\xymatrix @-4.9pc @! {
  & & & & & & & & & & & & & & & \\
  & & *+[o][F]{t_1} \ar[dr] \ar@{-}[ul] \ar@{--}[ul] \ar@{.}[ul]  \ar@/^0.25pc/[drrr] & & & & & & & & & & & & & \\
  & & & *+[o][F]{\widetilde{t}} \ar@{-}[ul] \ar@{--}[ul] \ar@{.}[ul] \ar@/_0.5pc/[rr] & & *+[o][F]{d} \ar@{~}[dddlll] \ar@{~}[uull] & & & & *{} \ar@{~}[dddrrr] \ar@{~}[uurr] & & & & & & \\
  & & & & *{} \ar@{-}[ul] & & & & & & & & & & & \\
  & & & & & & & & & & & *{} \ar@{-}[uuuurrrr] & &  & & \\
  *{} \ar@{-}[rrrrrr] & & *{} \ar@{~}[uull] & & & & (-n-4,-n-2) \ar@{.}[uuuuurrrrr] \ar@{--}[uuulll] \ar@{.}[uuulll]  \ar@{-}[rr] & & *+[o][F]{} \ar@{-}[rr] & & (-n-2,-n) \ar@{--}[uuuuurrrrr] \ar@{-}[rrrrr] & & \ar@{~}[uuurrr] & & & \\
}
$$
Similarly, by Lemma 2.7 of \citep{J}, each $t_2 \rightarrow d,$ with $t_2 = (-n-2,-n+k)$ for $k \geq 0$ factors through $\widetilde{t} \rightarrow d.$ This is shown below. 
$$ \centering \def\objectstyle{\scriptstyle}
  
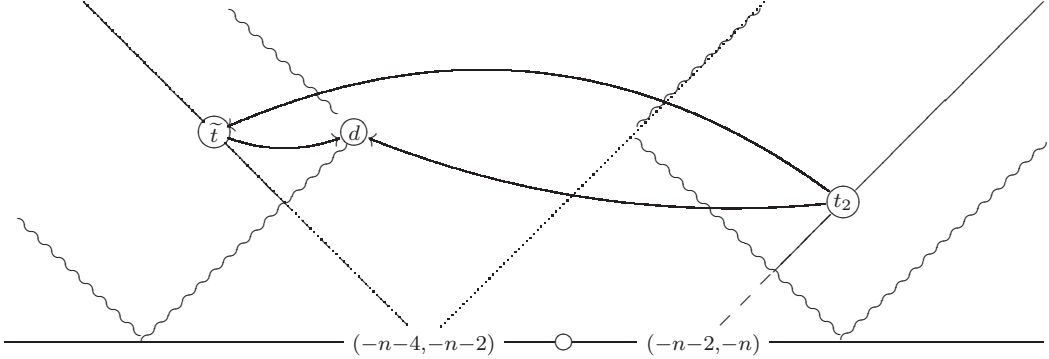
\xymatrix @-4.9pc @! {
  & & & & & & & & & & & & & & & \\
  & & *{} \ar@{-}[ul] \ar@{--}[ul] \ar@{.}[ul] & & & & & & & & & & & & & \\
  & & & *+[o][F]{\widetilde{t}} \ar@{-}[ul] \ar@{--}[ul] \ar@{.}[ul] \ar@/_0.5pc/[rr] & & *+[o][F]{d} \ar@{~}[dddlll] \ar@{~}[uull] & & & & *{} \ar@{~}[dddrrr] \ar@{~}[uurr] & & & & & & \\
  & & & & *{} \ar@{-}[ul] & & & & & & & & *+[o][F]{t_2} \ar@{-}[uuurrr] \ar@/^1pc/[ulllllll] \ar@/_3pc/[ulllllllll] & & & \\
  & & & & & & & & & & & *{} \ar@{-}[ur] & &  & & \\
  *{} \ar@{-}[rrrrrr] & & *{} \ar@{~}[uull] & & & & (-n-4,-n-2) \ar@{.}[uuuuurrrrr] \ar@{--}[uuulll] \ar@{.}[uuulll]  \ar@{-}[rr] & & *+[o][F]{} \ar@{-}[rr] & & (-n-2,-n) \ar@{--}[uurr] \ar@{-}[rrrrr] & & \ar@{~}[uuurrr] & & & \\
}
$$
Hence, $\widetilde{t} \rightarrow d$ is an almost-right-$\cT$-approximation of $d,$ because $\widetilde{t}$ deals with all but finitely many indecomposables of $\cT.$ This is because the wavy region intersects finitely many slices and coslices, and except for $(-n-2,\text{---})$ and $(\text{---},-n-2)$ each has only finitely many indecomposable $\cT$-objects. Therefore, by Lemma \ref{almostrightlem}, $d$ has a right-$\cT$-approximation. We have now dealt with each $d \in \ind \cD.$

Now, suppose $d$ is a hocolimit object. First we consider $d = \sE_{n+k}$ for $k \leq 0.$ This is illustrated below.
$$ \centering \def\objectstyle{\scriptstyle}
  
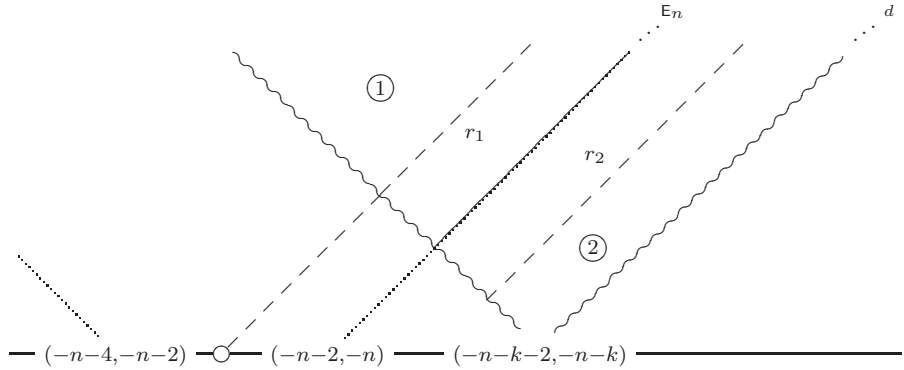
\xymatrix @-6.4pc @! {
& & & & & & & & & & & & & \hspace{4mm} {\iddots}^{\sE_{n}} & & & & \hspace{4mm} {\iddots}^{\hspace{1mm} d}  & \\
& & & & & & & & *+[o][F]{1} & & & & & & & & &\\
& & & & & & & & & & & & & & & & &\\
& & & & & & & & *{} \ar@{--}[uuurrr]_-{r_1} & & & & & & & & & \\
& & & & & & *{} & & & *{} \ar@{-}[uuuurrrr] & & & *+[o][F]{2} & & & &  \\
& & & & & & & & & & *{} \ar@{--}[uuuuurrrrr]^-{r_2} & & & & & &  \\
& *{} \ar@{-}[rr] & & (-n-4,-n-2) \ar@{-}[rr] \ar@{.}[uull] & & *+[o][F]{} \ar@{-}[rr] \ar@{--}[uuurrr] & & (-n-2,-n) \ar@{.}[uuuuuurrrrrr] \ar@{-}[rrrr] & & & & (-n-k-2,-n-k) \ar@{-}[rrrrrrr] \ar@{~}[uuuuuurrrrrr]  \ar@{~}[uuuuuullllll]& & & & & & &   \\
           }
$$
To clarify, region $\def\objectstyle{\scriptstyle}\xymatrix{*+[o][F]{1}}$ is the V-shaped region with the wavy line on the left-hand-side and the dashed line labelled $r_1$ on the right-hand-side (note, $r_1=(-n-3,\text{---})$. There are no $\cT$-objects in region $\def\objectstyle{\scriptstyle}\xymatrix{*+[o][F]{1}}$ due to the forbidden region property of the $\cT$-indecomposables which lie on the slice $(-n-2,\text{---}).$ Region $\def\objectstyle{\scriptstyle}\xymatrix{*+[o][F]{2}}$ is bounded by the dashed line labelled $r_2$ and the wavy lines (note, $r_2=(-n-1,\text{---})$. There are only finitely many $\cT$-objects in here, because region $\def\objectstyle{\scriptstyle}\xymatrix{*+[o][F]{2}}$ intersects finitely many slices each with finitely many $\cT$-indecomposables. The wavy lines illustrate the region of indecomposables which have nonzero morphisms to $d.$ Note, $\sE_n$ has nonzero morphisms to $d,$ and is included in this region. In all cases, the boundary lines are part of their respective regions. The solid line indicates the part of the slice $(-n-2,\text{---})$, together with its hocolimit $\sE_n,$ which has an infinite number of $\cT$-objects which map to $d$. Let $\sE_n \rightarrow d$ be a nonzero morphism. We claim $\sE_n \rightarrow d$ is an almost-right-$\cT$-approximation. Indeed, if $t \in \cT \cap (-n-2,\text{---})$ then $t \rightarrow d$ factorises as $t \rightarrow \sE_n \rightarrow d,$ by Lemma \ref{lem:fact2}.

Now let $d = \sE_{n+1}.$ The right-$\cT$-approximation is trivially zero, because $\hmD(\cT,d)=0$ since $d=\Sigma \sE_n \in \Sigma \cT$ and $\cT$ is weakly cluster tilting.

Finally, let $d=\sE_{n+k}$ for $k>1.$ The following illustrates this.
$$ \centering \def\objectstyle{\scriptstyle}
  
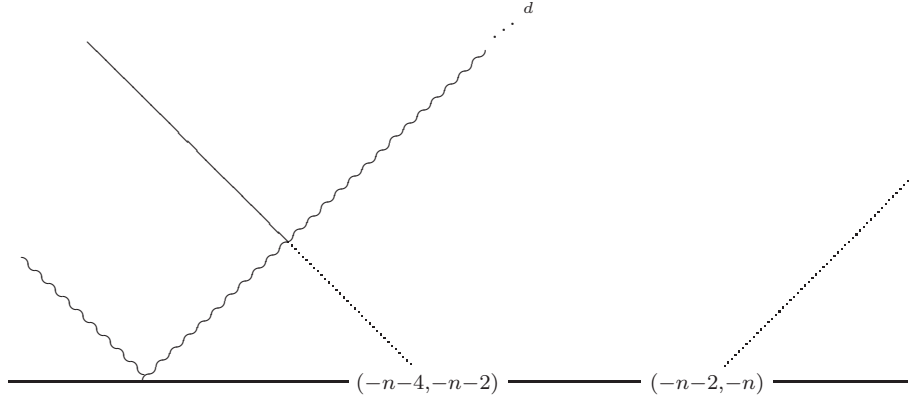
\xymatrix @-4.9pc @! {
& & & & & & & & \hspace{4mm} {\iddots}^{\hspace{1mm} d} & & & & & & \\
& & & & & &  & & & & & & & & \\
& & & & & & & & & & & & & & \\
& & & & & *{} \ar@{-}[uuulll] & & & & & & & & & \\
& & & & & & & & & & & & & & \\
& *{} \ar@{-}[rrrrr] & & *{} \ar@{~}[uull] \ar@{~}[uuuuurrrrr] & &  \ar@{-}[rr]  & & (-n-4,-n-2) \ar@{-}[rrr] \ar@{.}[uull] & &  \ar@{-}[rr] & & (-n-2,-n) \ar@{-}[rrr] \ar@{.}[uuurrr] & & & \\
}
$$
Pick any $\widetilde{t} \in \cT \cap (\text{---},-n-2),$ and a nonzero morphism $\widetilde{t} \rightarrow d,$ where $\widetilde{t} = (-n-k,-n-2).$ Then this is an almost-right-$\cT$-approximation of $d,$ for if $t \in \cT \cap (\text{---},-n-2)$ is equal to $(-n-s,-n-2),$ for $s \geq k,$ then $t \rightarrow d$ factors through $\widetilde{t} \rightarrow d,$ by Lemma \ref{lem:fact1}. Note, $\cT \cap (-n-2,\text{---})$ has only zero morphisms to $d.$ Therefore, there exists a right-$\cT$-approximation of $d.$

This covers all possibilities of what $d$ can be if $d$ is one of the hocolims. In every case, there exists a right-$\cT$-approximation of $d,$ so $\cT$ is right-approximating. By Lemma 3.2 of \citep{K}, this is enough to show that $\cT$ is also left-approximating. Hence, if $\fT$ has a fountain at $-n-2$, then $\cT$ is functorially finite.
\end{proof}

\begin{Lemma} \label{leapfrog2}
Let $\cT$ be a subcategory of $\Dbar$ associated with the arc diagram $\fT,$ where $\fT$ is a maximal, non-crossing, locally finite set of finite arcs. Then $\cT$ is not cluster tilting.
\end{Lemma}

\begin{proof}
If a category is weakly cluster tilting, but fails to be cluster tilting, then that category must fail to be functorially finite. This in turn amounts to the category failing to be both right and left-approximating. So we prove that $\cT$ is not right-approximating. Let $(p,q)$ be an arc in $\fT.$ By Lemma \ref{lem:locallyfinitemainclaim}, there exists a strong overarc $(p-\delta_1,q+\varepsilon_1) \in \fT$ where $\varepsilon_1, \delta_1 > 0.$ But this arc has an even longer arc enveloping it, $(p-\delta_2,q+\varepsilon_2),$ where $\varepsilon_2 > \varepsilon_1,$ and $\delta_2 > \delta_1.$ We are led to a sequence of arcs in $\fT$
$$(p,q), (p-\delta_1,q+\varepsilon_1), (p-\delta_2,q+\varepsilon_2), (p-\delta_3,q+\varepsilon_3), \cdots$$
where 
$$0 < \varepsilon_1 < \varepsilon_2 < \varepsilon_3 < \cdots$$
and
$$0 < \delta_1 < \delta_2 < \delta_3 < \cdots.$$

Now, if $(p,q) \in \fT,$ then this corresponds to $\Sigma^{-q}X_{q-p-2} \in \cT,$ and this can be rewritten as $\Sigma^{-p-2-i}X_{i}$ where $i = q-p-2$ (and note that $i \geq 0$ since $q \geq p+2$). Hence $(p,q)$ lies on the slice whose hocolimit is $\sE_{-p-2} \in \Dbar.$ We claim that $\sE_{-p-2}$ has no right-$\cT$-approximation. Let $(p-\delta,q+\varepsilon)$ be an arc in $\fT,$ where $\varepsilon, \delta >0.$ Then the corresponding indecomposable in $\cT$ is $\Sigma^{-q-\varepsilon}X_{q-p-2+\varepsilon+\delta}$ which can be rewritten as $\Sigma^{-p-2-j}X_k,$ where $j = -p-2+q+\varepsilon,$ and $k = -p-2+q+\varepsilon+\delta.$ Now, it is clear that $j \geq 0$ and $k \geq j.$ Forgiving the abuse of notation, this leads us to conclude that $(p-\delta,q+\varepsilon) \in \BW((p,q)),$ and therefore the indecomposable corresponding to the arc $(p-\delta,q+\varepsilon)$ is in the region of the Auslander-Reiten quiver which has nonzero maps to $\sE_{-p-2}.$ 

Now, let us again excuse the abuse of notation and consider 
$$H^{+}((p-1,q-1))=\{(a,b) \hspace{1.5mm} | \hspace{1.5mm} a \leq p-2 \hspace{1.5mm} \textrm{and} \hspace{1.5mm} p \leq b \leq q-2\}$$
and
$$H^{-}((p-1,q-1))=\{(a,b) \hspace{1.5mm} | \hspace{1.5mm} p \leq a \leq q-2 \hspace{1.5mm} \textrm{and} \hspace{1.5mm} q \leq b\}.$$ 
Clearly, $(p-\delta,q+\varepsilon) \not\in H^{+}((p-1,q-1))$, since it fails the second condition, and also $(p-\delta,q+\varepsilon) \not\in H^{-}((p-1,q-1))$ since it fails the first condition. Hence, 
$$(p-\delta,q+\varepsilon) \not\in H((p-1,q-1))$$ 
and so there are no nonzero morphisms from the indecomposable corresponding to the arc $(p,q)$ to the indecomposable corresponding to the arc $(p-\delta,q+\varepsilon).$ Similarly, there are no nonzero morphisms in the other direction, either.

The sequence 
$$(p,q), (p-\delta_1,q+\varepsilon_1), (p-\delta_2,q+\varepsilon_2), (p-\delta_3,q+\varepsilon_3), \cdots$$
gives an infinite sequence of indecomposables in $\cT$ all of which have nonzero morphisms to $\sE_{-p-2}.$ Let us group all of these together in the set $T.$ There are no nonzero morphisms from any object in the set $T$ to any other object in the set $T$, except for that same object itself. A right-$\cT$-approximation of $\sE_{-p-2}$ must be of the form $\tau : \widetilde{t_1} \oplus \cdots \oplus \widetilde{t_n} \rightarrow \sE_{-p-2}.$ But for such an object $\widetilde{t_1} \oplus \cdots \oplus \widetilde{t_n}$ to exist, it must have a summand $\widetilde{t_l}$ which allows an infinite number of $T$-indecomposables to factor through it - which is impossible. Hence, the $\Dbar$-object $\sE_{-p-2}$ has no right $\cT$-approximation, and therefore $\cT$ is not cluster tilting.
\end{proof}

\begin{Remark} \label{rem:onlyfountain}
We have shown that if a weakly cluster tilting subcategory $\cT$ of $\Dbar$ contains a hocolimit, then it contains at most one hocolimit. If it contains one hocolimit, (say, $\sE_n$) then $\cT$ is cluster tilting. If $\cT$ contains no hocolimits, then it cannot be cluster tilting. Hence, cluster tilting subcategories of $\Dbar$ have precisely one hocolimit. \demo
\end{Remark}

\noindent Remark \ref{rem:onlyfountain} motivates the following theorem.

\begin{Theorem} \label{thm:cts:maintheorem}
If a weakly cluster tilting subcategory $\cT$ of $\Dbar$ corresponds to a set of arcs $\fT,$ then $\cT$ is cluster tilting if and only if $\fT$ has an arc to infinity and a fountain. \demo
\end{Theorem}

  \bibliographystyle{is-plain}
  \bibliography{../references}  

\begin{thebibliography}{10}

\bibitem{BBM}
K.~Baur, A.B. Buan, and R.J. Marsh.
\newblock Torsion pairs and rigid objects in tubes.
\newblock {\em Algebras and Representation Theory}, {\bf{17(2)}}:\penalty0
  565--591, 2014.

\bibitem{BN}
M.~B{\"{o}}kstedt and A.~Neeman.
\newblock Homotopy limits in triangulated categories.
\newblock {\em Compositio Mathematica}, {\bf{86(2)}}:\penalty0 209--234, 1993.

\bibitem{BABHK}
A.B. Buan and H.~Krause.
\newblock Cotilting modules over tame hereditary algebras.
\newblock {\em Pacific Journal of Mathematics}, {\bf{211(1)}}:\penalty0 41--59,
  2003.

\bibitem{J}
T.~Holm and P.~J{\o}rgensen.
\newblock On a cluster category of infinite {D}ynkin type, and the relation to
  triangulations of the infinity-gon.
\newblock {\em Mathematische Zeitschrift}, {\bf{270}}:\penalty0 277--295, 2012.

\bibitem{AS}
L.~Angeleri H\"{u}gel and J.~S\'{a}nchez.
\newblock Tilting modules over tame hereditary algebras.
\newblock {\em Journal f\"{u}r die reine und angewandte Mathematik (Crelles
  Journal)}, {\bf{682}}:\penalty0 1--48, 2013.

\bibitem{I}
O.~Iyama.
\newblock Higher dimensional {A}uslander-{R}eiten theory on maximal orthogonal
  subcategories.
\newblock {\em Advances in Math.}, {\bf{210(1)}}:\penalty0 22--50, 2007.

\bibitem{JSM}
P.~J{\o}rgensen.
\newblock Spectra of modules.
\newblock {\em J. Algebra}, {\bf{244}}:\penalty0 744--784, 2001.

\bibitem{KYZ}
B.~Keller, D.~Yang, and G.~Zhou.
\newblock The {H}all algebra of a spherical object.
\newblock {\em Journal of the London Mathematical Society},
  {\bf{80(3)}}:\penalty0 771--784, 2009.

\bibitem{K}
S.~Koenig and B.~Zhu.
\newblock From triangulated categories to abelian categories: cluster tilting
  in a general framework.
\newblock {\em Mathematische Zeitschrift}, {\bf{258(1)}}:\penalty0 143--160,
  2008.

\bibitem{N}
A.~Neeman.
\newblock The {G}rothendieck duality theorem via {B}ousfield's techniques and
  {B}rown representability.
\newblock {\em Journal of the American Mathematical Society}, {\bf{9} Issue
  1}:\penalty0 205--236, 1996.

\bibitem{W}
C.A. Weibel.
\newblock {\em An introduction to homological algebra}.
\newblock Cambridge Studies in Advanced Mathematics, {\bf{38}}, Cambridge
  University Press, 1995.

\end{thebibliography}

\section*{Acknowledgements}
This work was supported in part by the EPSRC Centre for Doctoral Training in Data Science, funded by the UK Engineering and Physical Sciences Research Council (grant reference 1244024). The author is very grateful.

\end{document}